\documentclass[english,12pt]{amsart}
\usepackage{amssymb}
\usepackage{amsfonts}
\usepackage{amscd}
\usepackage[all]{xy}
\usepackage{color}
\usepackage{enumerate}
\usepackage[utf8]{inputenc}
\usepackage{relsize}
\usepackage{todonotes}

\usepackage[T1]{fontenc}
\usepackage{xr-hyper}
\usepackage[colorlinks=true, breaklinks=true, urlcolor= black, allcolors= black,bookmarksopen=true,linktocpage=true,plainpages=false,pdfpagelabels]{hyperref}
\usepackage{cleveref}

\CompileMatrices

\swapnumbers

\def\rad{\operatorname{rad}}
\def\radop{\rad_{\mathrm{op}}}
\def\radcl{\rad_{\mathrm{cl}}}
\def\radopc{\rad_{\mathrm{op},c}}

\def\radclc{\rad_{\mathrm{cl},c}}

\def\eqc{{\mathrm{ecc}}}  
\def\mix{{\mathrm{mix}}}

\def\hc{h$^{{\rm{coars}}}$}
\def\hmix{h$^\mix$}

\def\val{{\mathrm{val}}}

\def\eqc{{\mathrm{ecc}}}  

\def\11{{\mathbf 1}}

\def\QQ{{\mathbb Q}}
\def\RR{{\mathbb R}}

\def\cA{{\mathcal A}}
\def\cB{{\mathcal B}}

\def\cL{{\mathcal L}}
\def\cM{{\mathcal M}}

\def\cO{{\mathcal O}}

\def\cT{{\mathcal T}}

\mathchardef\alphag="7C0B \mathchardef\betag="7C0C
\mathchardef\gammag="7C0D \mathchardef\deltag="7C0E
\mathchardef\varepsilong="7C22 \mathchardef\varphig="7C27
\mathchardef\psig="7C20 \mathchardef\zetag="7C10
\mathchardef\epsilong="7C0F \mathchardef\rhog="7C1A
\mathchardef\taug="7C1C \mathchardef\upsilong="7C1D
\mathchardef\iotag="7C13 \mathchardef\thetag="7C12
\mathchardef\pig="7C19 \mathchardef\sigmag="7C1B
\mathchardef\etag="7C11 \mathchardef\omegag="7C21
\mathchardef\kappag="7C14 \mathchardef\lambdag="7C15
\mathchardef\mug="7C16 \mathchardef\xig="7C18
\mathchardef\chig="7C1F \mathchardef\nug="7C17
\mathchardef\varthetag="7C23 \mathchardef\varpig="7C24
\mathchardef\varrhog="7C25 \mathchardef\varsigmag="7C26
\mathchardef\Omegag="7C0A \mathchardef\Thetag="7C02
\mathchardef\Sigmag="7C06 \mathchardef\Deltag="7C01
\mathchardef\Phig="7C08 \mathchardef\Gammag="7C00
\mathchardef\Psig="7C09 \mathchardef\Lambdag="7C03
\mathchardef\Xig="7C04 \mathchardef\Pig="7C05
\mathchardef\Upsilong="7C07

\newtheorem{thm}[subsubsection]{Theorem}
\newtheorem{lem}[subsubsection]{Lemma}
\newtheorem{cor}[subsubsection]{Corollary}

{Addendum}

\theoremstyle{definition}
\newtheorem{defn}[subsubsection]{Definition}

\newtheorem{def-prop}[subsubsection]{Proposition-Definition}
\newtheorem{def-theorem}[subsubsection]{Theorem-Definition}
\newtheorem{def-lem}[subsubsection]{Lemma-Definition}

\theoremstyle{remark}

\newtheorem{question}[subsubsection]{Question}

\theoremstyle{plain}

\numberwithin{equation}{subsection}

\def\boxit#1#2{\setbox1=\hbox{\kern#1{#2}\kern#1}%
\dimen1=\ht1 \advance\dimen1 by #1 \dimen2=\dp1 \advance\dimen2 by
#1
\setbox1=\hbox{\vrule height\dimen1 depth\dimen2\box1\vrule}%
\setbox1=\vbox{\hrule\box1\hrule}%
\advance\dimen1 by .4pt \ht1=\dimen1 \advance\dimen2 by .4pt
\dp1=\dimen2 \box1\relax}

\renewcommand{\theequation}{\thesubsection.\arabic{equation}}

\mathchardef\alphag="7C0B \mathchardef\betag="7C0C
\mathchardef\gammag="7C0D \mathchardef\deltag="7C0E
\mathchardef\varepsilong="7C22 \mathchardef\varphig="7C27
\mathchardef\psig="7C20 \mathchardef\zetag="7C10
\mathchardef\epsilong="7C0F \mathchardef\rhog="7C1A
\mathchardef\taug="7C1C \mathchardef\upsilong="7C1D
\mathchardef\iotag="7C13 \mathchardef\thetag="7C12
\mathchardef\pig="7C19 \mathchardef\sigmag="7C1B
\mathchardef\etag="7C11 \mathchardef\omegag="7C21
\mathchardef\kappag="7C14 \mathchardef\lambdag="7C15
\mathchardef\mug="7C16 \mathchardef\xig="7C18
\mathchardef\chig="7C1F \mathchardef\nug="7C17
\mathchardef\varthetag="7C23 \mathchardef\varpig="7C24
\mathchardef\varrhog="7C25 \mathchardef\varsigmag="7C26
\mathchardef\Omegag="7C0A \mathchardef\Thetag="7C02
\mathchardef\Sigmag="7C06 \mathchardef\Deltag="7C01
\mathchardef\Phig="7C08 \mathchardef\Gammag="7C00
\mathchardef\Psig="7C09 \mathchardef\Lambdag="7C03
\mathchardef\Xig="7C04 \mathchardef\Pig="7C05
\mathchardef\Upsilong="7C07

\newcommand{\RV}{\mathrm{RV}}

\newcommand{\rv}{\operatorname{rv}}

\newcommand{\Th}{\operatorname{Th}}

\newcommand{\ltz}{\mathrel{<\joinrel\llap{\raisebox{-1ex}{$\scriptstyle{0}\mkern8mu$}}}}

\DeclareMathOperator*{\im}{{im}}

\newcommand{\grad}{\operatorname{grad}}

\newcommand{\Rtimes}{\mathlarger{\mathlarger{\rtimes }}}

\def\Jac{\operatorname{Jac}}

\definecolor{immi}{rgb}{0,.6,.1}

\newbox\removebox
\newcommand\remove[1]{%
\setbox\removebox=\ifmmode\hbox{$#1$}\else\hbox{#1}\fi%
\leavevmode
\rlap{\textcolor{blue}{\vrule height0.8ex depth-0.6ex width\wd\removebox}}%
\box\removebox
}
\long\def\bigremove#1{%
\par\setbox\removebox=\vbox{#1}%
\vbox{%
\vbox to0pt{\hbox{\tikz\draw[color=blue,thick] (0,0) -- (\wd\removebox,-\ht\removebox)  (\wd\removebox,0) -- (0,-\ht\removebox);}}
\box\removebox
}
}

\newcommand\acl{\mathrm{acl}}
\newcommand\dcl{\mathrm{dcl}}

\definecolor{orange}{rgb}{1,0.5,0}

\newcommand{\private}[1]{\leavevmode{\scriptsize\color{blue}\marginpar{{\scriptsize Private comment}}#1\par}}
\renewcommand{\private}[1]{}

\thanks{
The author was partially supported by KU Leuven IF C14/17/083, and partially by F.W.O. Flanders (Belgium) with grant number 11F1921N}

\title[Geometric criteria for $\ell$-h-minimality]{Hensel minimality: Geometric criteria for $\ell$-h-minimality}

\author[Vermeulen]{Floris Vermeulen}
\address{KU Leuven, Department of Mathematics, B-3001 Leu\-ven, Bel\-gium}
\email{floris.vermeulen@kuleuven.be}
\urladdr{https://sites.google.com/view/floris-vermeulen/homepage}

\begin{document}

\begin{abstract}
Recently, Cluckers, Halupczok and Rideau-Kikuchi developed a new axiomatic framework for tame non-Archimedean geometry, called Hensel minimality. It was extended to mixed characteristic together with the author. Hensel minimality aims to mimic o-minimality in both strong consequences and wide applicability. In this article, we continue the study of Hensel minimality, in particular focusing on $\omega$-h-minimality and $\ell$-h-minimality, for $\ell$ a positive integer. Our main results include an analytic criterion for $\ell$-h-minimality, preservation of $\ell$-h-minimality under coarsening of the valuation and $\ell$-dimensional geometry. 
\end{abstract}

\maketitle

\section{Introduction}

It is a classical result due to Tarski and Seidenberg that the theory of real closed fields (in the ring language) admits quantifier elimination and is complete. This in turn puts strong restrictions on the definable sets in this structure, and forms a typical example of \emph{tame behaviour}. Since then, many other theories have been shown to exhibit some form of tameness. For example, there is the theory of $\RR$ with restricted analytic functions \cite{vanDenDries.Ran}, or the theory of $\RR$ with a symbol for the global exponential $\exp$ \cite{Wilkie0}. For real closed fields, this idea of tame behaviour is very well axiomatized by the notion of o-minimality \cite{Pill-Stein, KnightPillSt, vdD}. Indeed, it is fair to say that o-minimality is the answer to Grothendieck's idea of ``topologie moder\'{e}e'' in the Archimedean setting. In recent years o-minimality has found various deep diophantine applications, for example towards the Andr\'{e}--Oort conjecture \cite{Pila}.

For non-Archimedean fields such as the $p$-adics $\QQ_p$ or fields of Laurent series $k((t))$, it has been known for a long time that similar forms of tame behaviour occur \cite{AK1, AK2, AK3, Cohen, Ersov, Mac}. These results formed the basis for many further developments in the model theory of valued fields \cite{Pas}, motivic integration \cite{DL, DLinvent, DLbarc}, Berkovich spaces \cite{HL} and point counting applications in number theory \cite{Denef}. Since the conception of o-minimality in the 80s, people have been looking for an axiomatic framework to explain many instances of tameness in the non-Archimedean setting. Mimicking the definition of o-minimality directly yields the notions of P-minimality \cite{Haskell} and C-minimality \cite{HM, Macpherson}. Geared towards integration there are the notions of b-minimality \cite{CLb} and V-minimality \cite{HK}. However, each of these notions has some shortcomings, either to their consequences, their generality, or the naturality of their axioms. The recent notion of Hensel minimality aims to fix this by providing a natural framework for tame non-Archimedean geometry with strong consequences \cite{CHR}. It provides a natural axiomatic setting in which to study the above mentioned applications, such as Pila--Wilkie point counting and Yomdin--Gromov parametrizations \cite{CHRV, CNV}, t-stratifications \cite{CH_tstrat} and motivic integration. Hensel minimality was developed by Cluckers, Halupczok and Rideau-Kikuchi in equicharacteristic zero and extended to mixed characteristic together with the author \cite{CHRV}. 

In more detail, Cluckers, Halupczok and Rideau-Kikuchi introduce a hierarchy of tameness notions: for each natural number $\ell$ (including zero) and for $\ell=\omega$, they introduce the notion of $\ell$-h-minimality. We use the term ``Hensel minimality'' or ``h-minimality'' to speak about these notions implicitly. The key notions here are $\omega$-h-minimality and 1-h-minimality, with $\omega$-h-minimality being the stronger of the two. Many classical theories such as the $p$-adics or fields of Laurent series are $\omega$-h-minimal (in the valued field language), but some more exotic examples are only known to be 1-h-minimal.

The aim of this article is to continue the study of h-minimality, focussing on $\omega$-h-minimality. This is done by studying the weaker notions of $\ell$-h-minimality for arbitrary $\ell\geq 1$, since $\omega$-h-minimality is equivalent to $\ell$-h-minimality for every $\ell$. The key result is a geometric criterion for $\ell$-h-minimality, based on definable functions on $\ell$-dimensional space. Such a criterion was given in \cite{CHR} for $\ell=1$. This allows us to extend and strengthen many results from \cite{CHR, CHRV}, which were only known for $1$-h-minimality, to $\ell$-h-minimal theories as well (for every $\ell$). In particular, we prove that $\ell$-h-minimality is preserved under $\RV$-enrichment and after coarsening of the valuation. This allows us to give a treatment of h-minimality in mixed characteristic as well. A second goal of this article is the further development of higher dimensional geometry under h-minimality. We again improve results from \cite{CHR, CHRV} assuming $\ell$-h-minimality and prove various new results. 

Our work on higher-dimensional geometry may be a first step towards Taylor approximation and parametrization results in all dimensions, refining and strengthening results from \cite{CHR} in dimension one. Such parametrizations are motivated by Diophantine applications on counting points of bounded height, where they form a key ingredient. In dimension one this already works well, see \cite{CHRV, CNV}. A second motivation is towards higher order and stronger results on $t$-stratifications, in the line of \cite{CH_tstrat}. Finally, our results answer several of the open questions in \cite{CHRV} and we propose several more questions in Section \ref{sec:further.questions}.

\subsection{Main results} Let us briefly discuss our main results. In \cite[Thm.\,2.9.1]{CHR} and \cite[Thm.\,2.2.7]{CHRV} an analytic criterion was given for 1-h-minimality. This shows that 1-h-minimality is of geometric nature. In essence, this criterion asserts that for every definable function $f: K\to K$ there is some finite definable set $C$, such that $f$ maps open balls disjoint from $C$ to open balls (or singletons), and scales the radii of such balls in a compatible manner. We provide a similar such analytic criterion for $\ell$-h-minimality for any $\ell\geq 1$. The criterion is very similar, but now for definable maps $f: K^\ell\to K$. The notion of open balls in $K^\ell$ doesn't make sense, and is instead replaced by certain twisted boxes in a cell decomposition. This result shows that also $\ell$-h-minimality is of geometric nature, up to $\ell$-dimensional objects. In particular, since $\omega$-h-minimality is equivalent to $\ell$-h-minimality for every integer $\ell$, this shows that also $\omega$-h-minimality is of geometric nature. A second main result is that $\ell$-h-minimality is preserved under $\RV$-expansion of the language, for any $\ell$. This extends \cite[Thm.\,4.1.19]{CHR}, where this was proven for $\ell=0,1$ or $\omega$.

We prove that $\ell$-h-minimality in equicharacteristic zero is preserved after arbitrary coarsening of the valuation, extending part of \cite[Thm.\,2.2.7]{CHRV}. Valuation coarsening allows us to develop a good working theory of $\ell$-h-minimality in mixed characteristic, as in \cite{CHRV}. 

We extend many other results from \cite{CHR} as well, in particular on higher-dimensional geometry. One of the main results here is a compatible domain-image preparation for definable functions $f: K^\ell\to K$ under $\ell$-h-minimality, extending \cite[Thm.\,5.2.4 Add.\,3]{CHR}. This results states that we can find a cell decomposition of the domain and image of $f$ which are compatible, meaning that $f$ maps twisted boxes in the domain to twisted boxes in the image. We prove a weaker version of this result under just $1$-h-minimality. We also prove a refined version of the supremum Jacobian property \cite[Thm.\,5.4.10]{CHR} for definable functions $f: K^\ell\to K$ under $\ell$-h-minimality. Finally, we prove a product of ideal preparation result. This theorem extends \cite[Thm.\,4.2.3]{CHR} to arbitrary ideals $I$. Moreover, under $\ell$-h-minimality, it shows that one can still prepare objects defined using more than $\ell$ parameters from $\RV_\lambda$, at the cost of some strength of preparation.
 
\subsection{Outline} In Section~\ref{sec:equivalent.condition} we briefly recall the main notions from~\cite{CHR} and~\cite{CHRV} about Hensel minimality. We state and prove our main result, giving an analytic criterion for $\ell$-h-minimality. We also prove that $\ell$-h-minimality is preserved under $\RV$-expansion of the language. In Section \ref{sec: mixed char} we define $\ell$-h-minimality in mixed characteristic and show that coarsening the valuation preserves $\ell$-h-minimality. We set up some basic ideas to transfer result from equicharacteristic zero to mixed characteristic. In Section \ref{sec:higher.dimension} we prove the aforementioned results about higher-dimensional geometry under $\ell$-h-minimality. We give mixed characteristic versions of many results as well. Finally, in Section \ref{sec:product.of.ideal.prep} we prove our product of ideal preparation. Section \ref{sec:further.questions} contains some natural follow-up questions related to this work.

\subsection{Acknowledgements} The author would like to thank Raf Cluckers for interesting discussions around this article, and many useful comments on an earlier draft. The author thanks Immanuel Halupczok for the many helpful comments on an earlier draft, in particular on Section \ref{sec:product.of.ideal.prep}. The author thanks the referee for a careful reading of this article.

\section{A criterion for $\ell$-h-minimality}\label{sec:equivalent.condition}

\subsection{Definitions and setting}\label{sec:defn.equivalences}

Let $\cT$ be a theory of characteristic zero valued fields in a language $\cL$ expanding the language of valued fields $\cL_\val = \{+, \cdot, \cO\}$. For most of this section we will restrict to equicharacteristic zero, but for now we allow the models of $\cT$ to have mixed characteristic as well.

Let us recall the set-up from \cite{CHR}. Let $K$ be a (non-trivial) valued field with valuation ring $\cO_K$. The maximal ideal will be denoted by $\cM_K$, while the value group is denoted by $\Gamma_K$. We use the multiplicative notation $|\cdot|: K\to \Gamma_K$ for the valuation. An open, resp. closed, ball in $K$ is a set of the form
\begin{align*}
B_{<\lambda}(a) = \{x\in K\mid |x-a| < \lambda\}, \quad B_{\leq \lambda}(a) = \{x\in K\mid |x-a|\leq \lambda\},
\end{align*}
for some $a\in K, \lambda\in \Gamma_K^\times$. Note that balls are always non-empty.

For $I$ a proper non-trivial ideal of $\cO_K$, let $\RV_I^\times$ be the quotient $K^\times/(1+I)$. The natural map $K^\times \to \RV_I^\times$ is denoted by $\rv_I$. We define $\RV_I = \RV_I^\times \cup\{0\}$ and extend the map $\rv_I$ to all of $K$ by $\rv_I(0) = 0$. If $I$ is the open ball $B_{<\lambda}(0)$ for some $\lambda\in \Gamma_K^\times$ we denote $\RV_I$ and $\rv_I$ by $\RV_\lambda$ and $\rv_\lambda$. Moreover, we let $\RV$ be $\RV_1$ and similarly for $\rv$. The map $\rv$ is called the \emph{leading term map}. If $m$ is a positive integer, we define for simplicity $\RV_m = \RV_{|m|}$. Note that if $K$ is of equicharacteristic zero, then $\RV_m = \RV$ for each positive integer $m$, since $|m|=1$. So $\RV_m$ will only play a role in mixed characteristic.

We recall what preparing subsets of $K$ means. Let $X$ be a subset of $K$ and let $\lambda \in \Gamma_K^\times$. We say that $X$ is \emph{$\lambda$-prepared} by a finite set $C$ if for all $x, x'\in K$ we have that if 
\[
\rv_\lambda(x-c) = \rv_\lambda(x'-c), \text{ for all }c\in C,
\] 
then either $x$ and $x'$ are both in $X$, or they are both not in $X$. In other words, the condition whether some $x\in K$ belongs to $X$ depends only on the tuple $(\rv_\lambda(x-c))_{c\in C}$. Note the similarity with o-minimality when $X\subset \RR$ is definable and $\rv$ is replaced by the sign function on $\RR$, see also the introduction of \cite{CHR}.

If $\xi$ is in $\RV_\lambda^\times$ and $c\in K$ then the set 
\[
c + \rv_\lambda^{-1}(\xi) 
\]
is said to be \emph{a ball $\lambda$-next to $c$}. If $C$ is a finite subset of $K$, then an open ball $B$ is said to be \emph{$\lambda$-next to $C$} if it is the intersection of balls $B_c$ for $c\in C$, where every ball $B_c$ is $\lambda$-next to $c$.

The condition of preparation can be restated as follows. A set $X\subset K$ is $\lambda$-prepared by the finite set $C$ if every ball $\lambda$-next to $C$ is either contained in $X$, or disjoint from $X$.

All of the above definitions also make sense for a proper ideal $I$ in $\cO_K$. In that case, we speak of a set being $I$-prepared by $C$, and of balls $I$-next to $C$. 

We recall from \cite{CHRV} the definition of $\ell$-\hmix-minimality.

\begin{defn}[{{\cite[Def.\,2.2.1]{CHRV}}}]
Let $\ell\geq 0$ be an integer or $\omega$. We say that $\cT$ is \emph{$\ell$-\hmix-minimal} if the following holds in every model $K$ of $\cT$. Let $n\geq 1$ be an integer, $\lambda$ be in $\Gamma_K^\times$, $\lambda\leq 1$, $A$ be a subset of $K$, $A'$ be a subset of $\RV_\lambda$ with $\#A'\leq \ell$ and let $X\subset K$ be $(A\cup \RV_n\cup A')$-definable. Then there exists an integer $m\geq 1$ such that $X$ is $|m|\lambda$-prepared by a finite $A$-definable set $C\subset K$. 
\end{defn}

If all models of $\cT$ have equicharacteristic zero, we speak simply of $\ell$-h-minimality. In that case, this definition simplifies, as $|m|=1$ for any non-zero integer $m$, and $\RV_n = \RV$. In general, $\lambda$-preparation in equicharacteristic zero becomes $|m|\lambda$-preparation for some positive integer $m$ in mixed characteristic, see also \cite[Ex.\,2.2.2]{CHRV}.

\subsection{A criterion for $\ell$-h-minimality}\label{sec:criterion}

For the rest of this section, we will assume that all models of $\cT$ have equicharacteristic zero.

In \cite[Thm.\,2.9.1]{CHR}, an analytic criterion for 1-h-minimal theories is given in equicharacteristic zero. This is further extended in \cite[Thm.\,2.2.7]{CHRV} to also include mixed characteristic. The criterion boils down to saying that any definable function $f: K\to K$ maps open balls $1$-next to some finite set $C$ to open balls (or singletons), and preserves their radii in a compatible manner, together with some basic preservation of dimension. We give a similar such analytic criterion for $\ell$-h-minimality ($\ell\geq 1$). Roughly speaking, the criterion asserts that the same holds for any definable function $f: K^\ell\to K$. The role of balls will be played by twisted boxes of a cell decomposition in higher dimension.

For $m\leq n$ denote by $\pi_{\leq m}$ (or $\pi_{<m+1}$) the projection $K^n\to K^m$ onto the first $m$ coordinates and by $\pi_m$ the projection onto the $m$-th coordinate.

\begin{defn} Let $\cT$ be a theory of equicharacteristic zero valued fields in a language expanding the language of valued fields and let $K$ be a model of $\cT$. Let $n\geq 1$ be an integer, let $\lambda\in \Gamma_K^\times$, $\lambda\leq 1$ and let $A\subset K\cup \RV$ be a parameter set. Let $X$ be an $A$-definable subset of $K^n$, $c_i: \pi_{<i}(X)\to K$ be $A$-definable maps and let $r\in \RV_\lambda^n$. Then $X$ is called an $A$-definable \emph{$\lambda$-twisted box} if 
\[
X = \{x\in K^n\mid (\rv_\lambda(x_i - c_i(\pi_{<i}(x)))_{i=1, ..., n} = r\}.
\]
The tuple $c=(c_i)_i$ is called the \emph{center tuple of $X$}. If $\lambda=1$, we speak simply of a \emph{twisted box}. We will use the notation $\Rtimes_\lambda (c, r)$ for the $\lambda$-twisted box $X$ as defined here, with \emph{center $c$ and value $r$}. For twisted boxes, we use the simpler notation $\Rtimes (c, r)$.

An $A$-definable \emph{decomposition} of $K^n$ is an $A$-definable map $\chi: K^n\to \RV^k$ for some integers $k\geq 1$ such that every fibre $\chi^{-1}(\xi)$ of $\chi$ is an $A\cup\{\xi\}$-definable twisted box whose center tuple moreover depends definably on $\xi$. When we speak about the \emph{twisted boxes} of a decomposition $\chi: K^n\to \RV^k$ we mean the (non-empty) fibres of the map $\chi$. A \emph{$\lambda$-twisted box} of a decomposition is a $\lambda$-twisted box contained in a fibre $\chi^{-1}(\xi)$ with the same center tuple. 

If $I$ is a proper ideal of $\cO_K$, we speak of \emph{$I$-twisted boxes} when we replace $\rv_\lambda$ by $\rv_I$.
\end{defn}

It is important to note that a decomposition consists of both the map $\chi: K^n\to \RV^k$ as well as the center tuples, definably parametrized by $\xi\in \RV^k$. However, we typically refer to the map $\chi$ as the decomposition and suppress the center tuples from notation. Note also that the value $r$ of a twisted box $\chi^{-1}(\xi)$ automatically depends definably on $\xi$. Indeed, if the center tuple of $\chi^{-1}(\xi)$ is $c_{\xi, 1}, \ldots, c_{\xi, n}$ then the value $r\in \RV^n$ of the twisted box $\chi^{-1}(\xi)$ is given by $(\rv(x_i - c_{\xi,i}(\pi_{<i}(x)))_{i=1, ..., n}$ for any $x\in K^n$ with $\chi(x) = \xi$. Also, if $F$ is a $\lambda$-twisted box of $\chi$ contained in $\chi^{-1}(\xi)$ then its value $r'\in \RV_\lambda^n$ always lies above the value $r$ under the component-wise projection map $\RV_\lambda^n\to \RV^n$.

The role of open balls next to some finite set will be played by the fibres of a decomposition in higher dimensions. Assuming 0-h-minimality, these concepts coincide in dimension one after a possible refinement, see Lemma~\ref{lem:decomp.dim.one}. 

We define our analytic criterion for $\ell$-h-minimality.

\begin{defn}
Let $\ell\geq 1$ be an integer and let $\cT$ be a theory of equicharacteristic zero valued fields. We say that $\cT$ satisfies (T$\ell$, D) if the following two properties are satisfied for each model $K$ of $\cT$ and each $A\subset K\cup \RV$.
\begin{itemize}
\item[(T$\ell)$] For any $A$-definable maps $f: K^\ell\to K$, $\psi: K^\ell\to \RV^{k'}$ there exists an $A$-definable decomposition $\chi: K^\ell \to \RV^k$ (for some integer $k\geq 1$) such that
\begin{enumerate}
\item $\psi$ is constant on the twisted boxes of $\chi$, 
\item for any twisted box $F$ of $\chi$, $f(F)$ is an open ball or a singleton, and
\item for any $\lambda\leq 1$ in $\Gamma_K^\times$ and any $\lambda$-twisted box $G$ of $\chi$ contained in the twisted box $F$ of $\chi$, $f(G)$ is either a singleton (if $f(F)$ is) or an open ball of radius $\lambda\radop f(F)$.
\end{enumerate}
We say that $f$ satisfies (T$\ell$) \emph{with respect to $\chi$} if condition 2 and 3 hold.
\item[(D)] For any $A$-definable map $f: K\to K$ the set
\[
\{y\in K\mid f^{-1}(y) \text{ is infinite}\}
\]
is a finite set.
\end{itemize}
\end{defn}

In dimension 1, property (T$\ell$) is almost the same as the valuative Jacobian property \cite[Lem.\,2.8.5]{CHR}. Indeed, in that case we may restate property (T$\ell$) by saying that there exists a finite set $C$ such that for every ball $B$ 1-next to $C$ there exists a $\mu_B\in \Gamma_K$ such that for every open ball $B'\subset B$ we have that $f(B')$ is an open ball of radius $\mu_B\radop B'$.

The following theorem is one of the main results of this article. It states that $\ell$-h-minimality is equivalent to property (T$\ell$, D), extending \cite[Thm.\,2.9.1]{CHR} to $\ell>1$.

\begin{thm}\label{thm: equivalences}
Let $\cT$ be a theory of equicharacteristic zero valued fields in a language $\cL$ expanding the language of valued fields. Let $\ell\geq 1$ be an integer. Then the following are equivalent.
\begin{enumerate}
\item $\cT$ is $\ell$-h-minimal.
\item $\cT$ satisfies (T$\ell$, D).
\end{enumerate}
\end{thm}
\begin{proof}
For $\ell=1$ this is Lemma~\ref{lem:1hmin.is.TellD} and Lemma~\ref{lem:TellD.is.1hmin}. For $\ell\geq 2$ this follows from Corollary~\ref{lem: ellhmin to Tell, D} and Lemma~\ref{lem: Tell,D to hmixmin}.
\end{proof}

Even though this result is only stated for $\ell$ an integer, it also implies a similar result for $\ell=\omega$. 

\begin{cor}
Let $\cT$ be a theory of equicharacteristic zero valued fields in a language $\cL$ expanding the language of valued fields. Then the following are equivalent.
\begin{enumerate}
\item $\cT$ is $\omega$-h-minimal.
\item $\cT$ satisfies (T$\ell$, D) for every $\ell\geq 1$.
\end{enumerate}
\end{cor}
\begin{proof}
The proof is immediate from Theorem \ref{thm: equivalences}, since $\omega$-h-minimality is equivalent to $\ell$-h-minimality for every $\ell\geq 1$.
\end{proof}

A second main result is that $\ell$-h-minimality is preserved under $\RV$-expansions. Let $K$ be a valued field, considered in some language $\cL$ expanding the language of valued fields. Recall from \cite[Sec.\,4]{CHR} that an \emph{$\RV$-expansion} (or \emph{$\RV$-enrichment}) of $\cL$ is an expansion $\cL'$ of the language $\cL$ by arbitrary predicates on cartesian powers of $\RV$. 

\begin{thm}\label{thm:preservation.RV.expansion}
Assume that $\Th_\cL(K)$ is $\ell$-h-minimal for some $\ell\geq 0$. Let $\cL'$ be an $\RV$-expansion of the language. Then $\Th_{\cL'}(K)$ is still $\ell$-h-minimal.
\end{thm}
\begin{proof}
For $\ell=0,1, \omega$ this is contained in \cite[Thm.\,4.1.19]{CHR}. For $\ell\geq 1$, this follows by combining Theorem \ref{thm: equivalences} and Theorem \ref{thm: RVexpansions}.
\end{proof}

The rest of this section is devoted to a proof of Theorem \ref{thm: equivalences}. 

\subsection{Proof of the equivalence for $\ell=1$}

Throughout this section, recall that $\cT$ is a theory of equicharacteristic zero valued fields in a language $\cL$ expanding the language of valued fields. Let $\ell\geq 1$ be an integer. We first prove the case $\ell=1$ of Theorem \ref{thm: equivalences}. To use results from \cite{CHR} in dimension one, we need the following lemma, relating decompositions with finite sets.

\begin{lem}\label{lem:decomp.dim.one}
Assume that $\cT$ is $0$-h-minimal and let $K$ be a model of $\cT$. Let $C$ be a finite $\emptyset$-definable set in $K$. Then there exists a $\emptyset$-definable decomposition $\chi: K\to \RV^n$ such that every twisted box of $\chi$ is contained in a ball $1$-next to $C$.
\end{lem}
\begin{proof}
This is essentially contained in \cite[Lem.\,2.5.4]{CHR}. For $x\in K$, define $\mu(x) = \min\{|x-c| \mid c\in C\}$, let $C(x) = \{c\in C\mid |x-c| = \mu(x)\}$ and define
\[
a(x) = \frac{1}{\# C(x)} \sum_{c\in C(x)} c.
\]
The map $a: K\to K$ has finite image, so by \cite[Lem.\,2.5.3]{CHR} there exists a $\emptyset$-definable injection $h: \im a\to \RV^k$ (for some integer $k$). Now define
\[
\chi: K\to \RV^{k+1}: x\mapsto (h(a(x)), \rv (x-a(x))).
\]
This is a decomposition with the desired property.
\end{proof}

We first prove that $1$-h-minimal theories satisfy (T1,D).

\begin{lem}\label{lem:1hmin.is.TellD}
Assume that $\cT$ is $1$-h-minimal, then $\cT$ satisfies (T1, D).
\end{lem}

\begin{proof}
Let $K$ be a model of $\cT$, let $A\subset K\cup \RV$ and let $f: K\to K, \psi: K\to \RV^k$ be $A$-definable. By the Jacobian property \cite[Cor.\,3.2.7]{CHR} there exists a finite $A$-definable set $C$ such that for any ball $B$ $1$-next to $C$, $f(B)$ is either an open ball or a singleton, and for any $\lambda\in \Gamma_K^\times, \lambda\leq 1$, and any open ball $B'$ $\lambda$-next to $C$ contained in $B$, $f(B')$ is either a singleton (if $f(B)$ is) or an open ball of radius $\lambda\radop B$. Using 1-h-minimality, we may enlarge $C$ so that $\psi$ is constant on balls 1-next to $C$. Let $\chi: K\to \RV^k$ be as provided by Lemma \ref{lem:decomp.dim.one} applied to $C$. Then property (T1) holds for $f$ with respect to $\chi$.

The fact that property (D) holds follows directly from~\cite[Lem.\,2.8.2]{CHR}.
\end{proof}

For the other direction, we first prove that a theory satisfying (T1, D) is $0$-h-minimal.

\begin{lem}\label{lem:TellD.is.0hmin}
Assume that $\cT$ satisfies (T1, D). Then $\cT$ is $0$-h-minimal.
\end{lem}

\begin{proof}
Let $A\subset K\cup \RV$ be a parameter set and let $X\subset K$ be $A$-definable. We want to $1$-prepare $X$ by a finite $(A\cap K)$-definable set. 

We first prove that we can $1$-prepare $X$ by a finite $A$-definable set. For this, note that property (D) implies that any $A$-definable map $f: \RV\to K$ has finite image. Hence, by using compactness we conclude the same for any $A$-definable map $f: \RV^k\to K$. By applying (T1) to the characteristic function of $X$, we obtain an $A$-definable decomposition $\chi: K\to \RV^k$ such that every fibre of $\chi$ is either contained in $X$ or disjoint from $X$. For $\xi\in \RV^k$, let $c_\xi\in K$ be the center of $\chi^{-1}(\xi)$. These form an $A$-definable family $(c_\xi)_{\xi\in \RV^k}$. By the argument above, the union $\cup_\xi \{c_\xi\} = C$ is then a finite $A$-definable set which $1$-prepares $X$. 

The previous paragraph shows that~\cite[Ass.\,2.5.1]{CHR} is satisfied, i.e.\ every $A$-definable subset of $K$ can be $1$-prepared by a finite set (without definability conditions). In particular, uniform boundedness as stated in~\cite[Lem.\,2.5.2]{CHR} holds for $\cT$. By an identical argument as in~\cite[Lem.\,2.9.4]{CHR} (which uses this uniform boundedness), we obtain that if $(C_\xi)_{\xi\in \RV^k}$ is a definable family of finite set $C_{\xi}\subset K$, then the union $\cup_\xi C_{\xi}$ is still finite. Considering the $A$-definable set $C$ obtained above as an $(A\cap K)$-definable family $(C'_{\zeta})_{\zeta\in \RV^m}$, we obtain that the $(A\cap K)$-definable set $C' = \cup_\zeta C'_\zeta$ is finite. But it also $1$-prepares $X$, which concludes the proof.
\end{proof}

\begin{lem}\label{lem:TellD.is.1hmin}
Assume that $\cT$ satisfies (T1, D). Then $\cT$ is $1$-h-minimal.
\end{lem} 
\begin{proof}
Let us say that $\cT$ satisfies property (T1') if the following holds for any model $K$ of $\cT$ and every $A\subset K\cup \RV$.
\begin{enumerate}
\item[(T1')] For any $A$-definable map $f: K\to K$ there exists a finite $A$-definable set $C$ such that for every ball $B$ $1$-next to $C$ there exists a $\mu_B\in \Gamma_K$ such that for all $x_1, x_2\in B$ we have
\[
|f(x_1)-f(x_2)| = \mu_B|x_1-x_2|.
\]
\end{enumerate}
In \cite[Thm.\,2.9.1]{CHR}, it is proved that $\cT$ is 1-h-minimal if and only if it satisfies property (T1') and property (D). (In \cite{CHR} these properties are called (T1) and (T2), but we rename them to avoid notational issues with property (T$\ell$).) Hence it is enough to prove that condition (T1, D) implies condition (T1', D).


Let $f: K\to K$ be $A$-definable for some $A\subset K\cup \RV$. Apply (T1) to obtain an $A$-definable decomposition $\chi: K\to \RV^k$ such that $f$ satisfies (T1) with respect to $\chi$. Let $c_\xi\in K$ be the center of $\chi^{-1}(\xi)$ for $\xi\in \RV^k$. By definition, this is an $A$-definable family $(c_\xi)_\xi$ with a parameter $\xi\in \RV^k$. By 0-h-minimality of $\cT$ and~\cite[Cor.\,2.6.7]{CHR}, the resulting union $C = \cup_\xi \{c_\xi \}$ is a finite $A$-definable set. Via an argument identical to~\cite[Lem.\,2.8.2]{CHR} which uses 0-h-minimality and property (D), we may enlarge $C$ such that $f$ is constant or injective on balls $1$-next to $C$. Let $B$ be a ball $1$-next to $C$. Then $B$ is contained in a twisted box $F$ of $\chi$. Because we are in dimension 1, $B$ is then a $\lambda$-twisted box of $\chi$, for some $\lambda\in \Gamma_K^\times, \lambda\leq 1$. Define $\mu_B$ as $\radop f(B) / \radop B$. By (T1) every open ball $B'$ contained in $B$ is mapped under $f$ to an open ball of radius $\mu_B \radop B'$. For $x, x'\in B$ this yields that
\[
|f(x)-f(x')|\leq \mu_B|x-x'|.
\]
Using injectivity of $f$, the same argument applied to the inverse of $f$ yields that 
\[
|f(x)-f(x')| = \mu_B|x-x'|. \qedhere
\]
\end{proof}

So in what follows we will assume that $\ell\geq 2$. In particular, since each of the statements in Theorem \ref{thm: equivalences} implies the one for $\ell-1$, we already know that $\cT$ is 1-h-minimal. 

\subsection{A useful lemma}

In this section we prove a little lemma allowing us to replace a definable family of finite sets $(C_q)_{q\in Q}$ (over some definable set $Q$) by a definable family of functions $(g_\eta)_{\eta\in \RV^k}$ parametrized over $\RV$ instead. This is frequently useful since preparation results typically produce families of finite sets, which we want to control somehow. Replacing this family of finite sets by a family of functions gives us more flexibility, as we have e.g.\ the Jacobian property~\cite[Lem.\,2.8.5]{CHR} or property (T$\ell$) at our disposal which we can apply to these functions.

\begin{lem}\label{lem:finite.set.to.functions}
Assume that $\cT$ is $0$-h-minimal and let $Q$ be an arbitrary (possibly imaginary) $\emptyset$-definable set. Let $C_q\subset K$ be a $\emptyset$-definable family of finite sets, for $q\in Q$. Then there exists a $\emptyset$-definable family of sets $Y_\eta\subset Q$, parametrized by $\eta\in \RV^k$, together with a $\emptyset$-definable family of functions $g_\eta: Y_\eta\to K$ such that
\[
\bigcup_{q\in Q} \{q\}\times E_q = \bigsqcup_{\eta\in \RV^k}  \operatorname{graph} g_\eta \subset Q\times K.
\]
\end{lem}

Since $0$-h-minimality is preserved under adding constants from $K\cup \RV$, this lemma also holds if everything is $A$-definable instead, for some $A\subset K\cup \RV$.

\begin{proof}
Using~\cite[Lem.\,2.5.3]{CHR} there exists a $\emptyset$-definable family of injections
\[
h_q: E_q\to \RV^k,
\]
for some fixed integer $k$. For $\eta\in \RV^k$, we define
\begin{align*}
Y_\eta &= \{q\in Q \mid \eta\in \im h_q\} \\
g_\eta &: Y_\eta \to K: q\mapsto h^{-1}_q(\eta).
\end{align*}
Then we indeed have that
\[
\bigcup_{q\in Q} \{q\}\times E_q = \bigsqcup_{\eta\in \RV^k}  \operatorname{graph} g_\eta. \qedhere
\]
\end{proof}

\subsection{From $\ell$-h-minimality to (T$\ell$, D)}

We prove the direction from $\ell$-h-minimality to property (T$\ell$, D), for $\ell\geq 2$. For later use, we prove a more general result about obtaining property (T$\ell$) for finitely many functions at once and, even more generally, for definable families over $\RV$.

\begin{lem}\label{lem: ellhmin to Tell, D}
Assume that $\cT$ is $\ell$-h-minimal and let $K$ be a model of $\cT$. Let $(f_\theta)_\theta$ be a $\emptyset$-definable family of functions $f_\theta: K^\ell\to K$, with $\theta\in \RV^n$. Then there exists a $\emptyset$-definable decomposition $\chi$ of $K^\ell$ such that for every $\theta\in \RV^n$, $f_\theta$ satisfies (T$\ell$) with respect to $\chi$.
In particular, the theory $\cT$ satisfies (T$\ell$, D).
\end{lem}

The notation $\RV_\bullet$ will be used to denote the disjoint union $\cup_{\lambda\leq 1} \RV_\lambda$, see \cite[Sec.\,2.6]{CHR}.

The proof is rather technical, so we give some intuition first. The idea is to induct on the dimension $\ell$. So for simplicity, assume that $\ell=2$ and that we just want to obtain property (T$\ell$) for one function $f: K^2\to K$. For $x\in K$ let $f_x: K\to K: y\to f(x,y)$. By induction, we can find for every $x$ a decomposition $\omega_x$ of $K$ such that $f_x$ satisfies property (T1) with respect to $\omega_x$. For every $x\in K$, we can now find a finite set $E_x \subset K$ preparing the image of $f_x$, i.e. $E_x$ 1-prepares the image of every twisted box of $\omega_x$ under $f_x$. For simplicity let's assume that $E_x$ is a singleton $e(x)$ and that the map $e: K\to K: x\mapsto e(x)$ is definable. By preparing several objects we will obtain a decomposition $\omega$ of $K^2$ such that the following holds for every twisted box $F$ of this decomposition. Let $B = \pi_1(F)$, which is an open ball.
\begin{enumerate}
\item The function $e$ satisfies the Jacobian property on $B$.
\item There exists a set $G\subset \RV$ (depending only on $F$) such that for every $x\in B$ 
\[
f (F\cap (\{x\}\times K)) = e(x) + \bigcup_{\xi\in G} \rv^{-1}(\xi).
\]
In other words, we have a good description of the image of $f$ on every vertical slice of $F$.
\item The set $\bigcup_{\xi\in G} \rv^{-1}(\xi)$ is an open ball (or a singleton).
\end{enumerate}
Now we have that 
\[
f(F) = \bigcup_x e(x) + \bigcup_{\xi\in G} \rv^{-1}(\xi) = e(B) + \bigcup_{\xi\in G} \rv^{-1}(\xi),
\]
which is the sum of two open balls, so it is again an open ball. (Or if both are singletons, then so is $f(F)$.) To have the correct scaling for $\lambda$-twisted boxes contained in $F$, we use the same argument but keeping track of $\RV_\lambda$-parameters as well.

\begin{proof} We want to find a $\emptyset$-definable decomposition $\chi: K^\ell\to \RV^m$ such that for every $\theta\in \RV^n$ we have that
\begin{enumerate}
\item for every twisted box $F$ of $\chi$, $f_\theta(F)$ is a singleton or an open ball, and
\item for every $\lambda\in \Gamma_K^\times, \lambda\leq 1$ and every $\lambda$-twisted box $G$ contained in a twisted box $F$ of $\chi$, $f_\theta(G)$ is either a singleton (if $f_\theta(F)$ is), or an open ball of radius $\lambda\radop f_\theta(F)$.
\end{enumerate}

We induct on $\ell$, the base case being $\ell=1$. In that case, use \cite[Lem.\,2.8.5]{CHR} to obtain finite $\emptyset$-definable sets $C_\theta$ for every $\theta\in \RV^n$ such that $f_\theta$ satisfies the valuative Jacobian property on balls 1-next to $C_\theta$. By compactness, we can assume that $(C_\theta)_\theta$ is a $\emptyset$-definable family. Since $\RV$-unions stay finite \cite[Cor.\,2.6.7]{CHR}, the union $C = \cup_\theta C_\theta$ is a finite $\emptyset$-definable set. Let $\chi$ be the $\emptyset$-definable decomposition from Lemma \ref{lem:decomp.dim.one} applied to $C$. Then every $f_\theta$ satisfies property (T1) with respect to $\chi$. This proves the case $\ell=1$.

So let $\ell>1$. For $x\in K, \theta\in \RV^n$ write $f_{\theta, x}: K^{\ell-1}\to K: y\mapsto f_\theta(x,y)$. Using induction for every $x\in K$, together with compactness, we take a $\emptyset$-definable family of decompositions $\omega_x: K^{\ell-1}\to \RV^p$ such that the family $(f_{\theta, x})_\theta$ satisfies property (T$\ell-1$) with respect to $\omega_x$. For $\eta\in \RV^p$, let $c_{x, \eta}$ be the center tuple of the twisted box $\omega_x^{-1}(\eta)$, which is also a $\emptyset$-definable family. Use 0-h-minimality \cite[Prop.\,2.6.2]{CHR} to find a $\emptyset$-definable family of finite sets $(E_x)_x$, $E_x\subset K$ such that $E_x$ $1$-prepares $f_{\theta, x}(\omega_x^{-1}(\eta))$ for every $\theta\in \RV^n$ and every $\eta\in \RV^p$. We apply Lemma~\ref{lem:finite.set.to.functions} to replace these finite sets by a family of functions. In more detail, we obtain an $\emptyset$-definable family $(Y_\eta)_{\eta\in \RV^k}$ of subsets $Y_\eta\subset K$, together with a $\emptyset$-definable family of maps $g_\eta: Y_\eta\to K$ such that
\[
\bigcup_{x\in K}\{x\}\times E_x = \bigsqcup_{\eta\in \RV^k} \operatorname{graph} g_\eta.
\]

We now use \cite[Prop.\,2.6.2, Lem.\,2.8.5]{CHR} to find a finite $\emptyset$-definable set $D$ with the following properties.
\begin{enumerate}
\item The set $D$ $1$-prepares all sets $Y_\eta$ for $\eta\in \RV^k$. This is possible by 0-h-minimality.

\item The maps $g_\eta$ for $\eta\in \RV^k$ satisfy the Jacobian property \cite[Lem.\,2.8.5]{CHR} on balls $1$-next to $D$. Note in particular that $g_\eta$ maps open balls disjoint from $D$ to open balls, and preserves their radii in the desired way. This is possible using 1-h-minimality.

\item The set $D$ uniformly $1$-prepares the $\emptyset$-definable set
\[
G_2 = \{ (x, \eta', \zeta) \in K\times \RV^p\times \RV^{\ell-1}\mid \omega_x^{-1}(\eta') = \Rtimes(c_{x, \eta'}, \zeta) \}.
\]
Recall that $\Rtimes(c_{x, \eta'}, \zeta)$ is the twisted box with center $c_{x, \eta'}$ and value $\zeta$. We prepare this set so that we can easily describe the twisted boxes in our final decomposition. This can be done using 0-h-minimality.

\item We also ensure that $D$ uniformly 1-prepares the set
\begin{align*}
G_3 = \{(x, \theta, \eta', \zeta, \eta)\in  K &\times \RV^n \times \RV^p\times \RV^{\ell-1}\times \RV^k\mid \\ x\in Y_\eta, \omega_x^{-1}(\eta') = \Rtimes(c_{x, \eta'}, \zeta), 
&\text{ and } f_{\theta, x}(\omega_x^{-1}(\eta')) \text{ is 1-prepared by }g_\eta(x)\}.
\end{align*}
We do this so that we can easily describe what the image of our twisted boxes under $f$ will be, which in turn will allow us to prove that these are open balls (or singletons). Again, this can be done using 0-h-minimality, where we use that preparing is a definable condition \cite[Lem.\,2.4.2]{CHR}.

\item Finally, we make sure that $D$ uniformly prepares the set 
\begin{align*}
G_4\subset K\times \RV^n\times \RV^p\times \RV^k\times \RV^{\ell-1}_\bullet\times \RV_\bullet
\end{align*}
consisting of those $(x, \theta, \eta', \eta, \zeta, \xi)$ in $K\times \RV^n\times \RV^p\times \RV^k\times \RV^{\ell-1}_\lambda\times \RV_\lambda$ such that $x\in Y_\eta$ and such that 
\[
g_\eta(x)+\rv^{-1}_\lambda(\xi) \subset f_{\theta, x}(\Rtimes_\lambda( c_{x, \eta'}, \zeta)).
\]
For fixed $(x, \theta, \eta', \eta, \zeta)$ we use the notation $G_4(x, \theta, \eta', \eta, \zeta)$ to mean the set of $\xi\in \RV_\lambda$ for which $(x, \theta, \eta', \eta, \zeta, \xi)$ is in $G_4$. This is the key preparation required to ensure that $f$ will map $\lambda$-twisted boxes in our final decomposition to open balls of the correct radius. For this we need $\ell$-h-minimality. 
\end{enumerate}

Use Lemma \ref{lem:decomp.dim.one} to obtain a $\emptyset$-definable decomposition $\chi_0: K\to \RV^{k'}$ such that every twisted box of $\chi_0$ is contained in a ball $1$-next to $D$. Define the decomposition
\[
\chi: K^\ell\to \RV^{k'}\times \RV^p: (x,y)\mapsto (\chi_0(x), \omega_x(y)).
\]
We claim that this is the desired decomposition.

Let $F$ be a twisted box of this decomposition and fix $\theta\in \RV^n$. Then $\pi_1(F)$ is an open ball $B$ which is contained in a ball $1$-next to $D$. Thus $F$ is of the form
\[
F = \bigcup_{x\in B} \{x\}\times \omega_x^{-1}(\eta'),
\]
for some fixed $\eta'\in \RV^p$. By preparation of $G_2$, there is a $\zeta\in \RV^{\ell-1}$ depending only on $F$ such that $\omega_x^{-1}(\eta')$ is precisely $\Rtimes(c_{x, \eta'}, \zeta)$ for any $x\in B$. By preparation of $G_3$ there exists an $\eta\in \RV^k$ depending only on $F$ and $\theta$ such that $f_{\theta,x} (\Rtimes(c_{x, \eta'}, \zeta))$ is 1-prepared by $g_\eta(x)$, for every $x\in B$. Hence we have that
\begin{align*}
f_\theta(F) &= f_\theta\left(\bigcup_{x\in B}\{x\}\times \Rtimes(c_{x, \eta'}, \zeta)\right) \\
	&= \bigcup_{x\in B} f_{\theta, x} (\Rtimes(c_{x, \eta'}, \zeta)) \\
	&= \bigcup_{x\in B} \bigcup_{\xi\in G_4(x, \theta, \eta', \eta, \zeta)} g_{\eta}(x)+\rv^{-1}(\xi) \\
	&= \bigcup_{x\in B} \left( g_\eta(x) + \bigcup_{\xi\in G_4(x, \eta', \eta, \zeta)} \rv^{-1}(\xi)\right).
\end{align*}
Now the set $\bigcup_{\xi\in G_4(x, \theta, \eta', \eta, \zeta)} \rv^{-1}(\xi)$ is independent of $x$, as $x$ runs over $B$ by preparation of $G_4$. Moreover, by definition of $G_4$ it is a translate of $f_{\theta, x} (B_x)$ for some twisted box $B_x$ of the decomposition $\omega_x$ (for some $x\in B$). Therefore, this set is an open ball (or a singleton), by induction. But also $g_{\eta}(B)$ is an open ball or a singleton, by the Jacobian property for $g_{\eta}$. If $x\in B$, then we conclude that 
\[
f_\theta(F) = g_{\eta}(B) + \bigcup_{\xi\in G_4(x, \theta, \eta', \eta, \zeta)} \rv^{-1}(\xi)
\]
is an open ball or a singleton, being the sum of two open balls or singletons.

For the final statement, if $F'$ is a $\lambda$-twisted box inside $F$ then the above reasoning goes through without change to show that
\[
f_\theta(F') = g_{\eta}(B') + \bigcup_{\xi'\in G_4(x,\theta, \eta', \eta, \zeta')} \rv^{-1}_\lambda(\xi')
\]
Here $B'$ is a ball contained in a ball $\lambda$-next to $D$ which is a subset of $B$ with $\radop B' = \lambda\radop B$, while $\zeta'$ and $\xi'$ are in $\RV_\lambda$. By the Jacobian property for $g_{\eta}$, the first set is a singleton or an open ball of radius $\lambda \radop g_{\eta}(B)$. Similarly, the second set 
\[
\bigcup_{\xi'\in G_4(x, \theta, \eta', \eta, \zeta')} \rv^{-1}_{\lambda}(\xi')
\]
is a translate of $f_{\theta, x}(B'_x)$ for some $\lambda$-twisted box $B'_x$ of the decomposition $\omega_x$ and contained in the twisted box $B_x$ as constructed above. Hence, this set is a singleton or an open ball of radius 
\[
\lambda\radop \left(\bigcup_{\xi\in G(x,\theta, \eta', \eta, \zeta)} \rv^{-1}(\xi)\right).
\]
This yields the desired conclusion.

To conclude property (T$\ell$, D) from this, note that a $\emptyset$-definable function $\psi: K^\ell\to \RV^p$ can be represented as a family of functions
\[
\psi_\zeta: K^\ell\to K: x\mapsto \begin{cases} 1 & \text{ if } \psi(x) = \zeta, \\
	0 & \text{ else,} \end{cases}
\]
for $\zeta\in \RV^p$. If $\chi$ is a decomposition such that every $\psi_\zeta$ satisfies (T$\ell$) with respect to $\chi$ then $\psi$ is constant on the twisted boxes of $\chi$.
\end{proof}

\subsection{Preservation under $\RV$-expansion}\label{sec: preservation under RV expansions}

Recall that the theory $\cT$ has \emph{algebraic Skolem functions} if for every model $K$ of $\cT$ and every $A\subset K$, we have $\acl_K(A) = \dcl_K(A)$. Equivalently, if for every model $K$ of $\cT$, all integers $n\geq 0, m>0$ and every $\emptyset$-definable $X\subset K^{n+1}$ for which the projection $\pi_{\leq n}$ has fibres of size exactly $m$ on $X$, there exists a $\emptyset$-definable function $f: K^n\to K$ whose graph is contained in $X$. We will abbreviate this condition by saying that $\acl=\dcl$ (in $\cT$).

To prove the other implication in Theorem \ref{thm: equivalences}, we pass through a theory where we have $\acl=\dcl$. To do so, we must prove that (T$\ell$, D) is preserved by $\RV$-expansions. Let $\cT$ be a theory satisfying (T$\ell$, D) in a language $\cL$ expanding the language of valued fields. Let $K$ be a model of $\cT$. Recall that an \emph{$\RV$-expansion} of $\cL$ is an expansion $\cL'\supset \cL$ by predicates on cartesian powers of $\RV_K$.

\begin{thm}\label{thm: RVexpansions}
Let $\ell\geq 1$ be an integer and let $\Th_\cL(K)$ satisfy property (T$\ell$, D) for some language $\cL$ expanding the language of valued fields. Let $\cL'$ be an $\RV$-expansion of $\cL$. Then $\Th_{\cL'}(K)$ still has property (T$\ell$, D).
\end{thm}

Note that, once we prove Theorem~\ref{thm: equivalences}, we also know that $\RV$-expansions preserve $\ell$-h-minimality.

\begin{proof}
We already know this for $\ell=1$, by~\cite[Thm.\,4.1.19]{CHR} together with Lemma~\ref{lem:1hmin.is.TellD} and Lemma~\ref{lem:TellD.is.1hmin}. In particular, we know that $\Th_{\cL'}(K)$ is 1-h-minimal. So let $\ell>1$. We follow the same strategy as in the proof of \cite[Thm.\,4.1.19]{CHR}. 

Let $f: K^\ell\to K$ and $\psi_0: K^\ell\to \RV^p$ be $\cL'(A)$-definable, for some $A\subset K\cup \RV$. By \cite[Cor.\,4.1.17]{CHR} we have that $f(x) \in \acl_{\cL,K}(A\cup\{x\})$. Hence we can find an $\cL(A)$-definable family of finite sets $C_x, x\in K^\ell$ such that $f(x)\in C_x$. By~\cite[Lem.\,2.5.3]{CHR} there exists an $\cL(A)$-definable family of injections $g_x: C_x\to \RV^k$ for some integer $k$. Define the maps
\begin{align*}
h &: K^\ell\to \RV^k: x\mapsto g_x(f(x)), \\
\tilde{f} &: K^\ell\times \RV^k\to K: (x,\xi)\mapsto \begin{cases}
		g_x^{-1}(\xi) &\text{ if } \xi\in g_x(C_x), \\
		0 &\text{ else}.
	\end{cases}
\end{align*}
Note that $h$ is $\cL'(A)$-definable while $\tilde{f}$ is $\cL(A)$-definable. Also, for any $x\in K^\ell$ we have $f(x) = \tilde{f}(x,h(x))$.

By \cite[Lem.\,4.3.4]{CHR} we obtain an $\cL(A)$-definable map $\psi_1: K^\ell\to \RV^q$ such that $h$ and $\psi_0$ are constant on the fibres of $\psi_1$. Using property (T$\ell$), we may then replace $\psi_1$ by an $\cL(A)$-definable decomposition $\psi: K^\ell\to \RV^m$ such that $h$ and $\psi_0$ are still constant on the fibres of $\psi$. Using compactness, we can find a $\cL(A)$-definable family over $\RV^k$ of decompositions $\chi_\xi: K^\ell\to \RV^m$ refining $\psi$ and such that the map $x\mapsto \tilde{f}(x,\xi)$ satisfies property (T$\ell$) with respect to $\chi_\xi$. Now consider the $\cL'(A)$-definable map
\[
\theta: K^\ell\to \RV^n\times \RV^m: (\chi_{h(x)}(x), \psi(x)).
\]
We claim that $\theta$ is a decomposition of $K^\ell$ such that $f$ satisfies property (T$\ell$) with respect to $\theta$. Let $F$ be the fibre of $\theta$ above $(\zeta, \xi')$ for some $(\zeta, \xi')\in \RV^n\times \RV^m$. Assume that $F$ is non-empty. Then $h$ is constant on $F$, say $\xi=h(x)$ for any $x\in F$. Then clearly $F\subset \chi_\xi^{-1}(\zeta)$. For the other inclusion, let $x_0\in F\subset \chi_\xi^{-1}(\zeta)$. Since $\psi$ is constant on fibres of $\chi_\xi$ we must have that $\psi(x) = \psi(x_0) = \xi'$ for any $x\in \chi_\xi^{-1}(\zeta)$. Therefore $x$ belongs to $F$. Hence any fibre of $\theta$ is a twisted box, with center parametrized by $(\xi, \zeta)$, so that $\theta$ is a decomposition. Finally, note that $f$ satisfies property (T$\ell$) on the twisted boxes of $\theta$. This proves that $\Th_{\cL'}(K)$ satisfies property (T$\ell$, D).
\end{proof}

\begin{cor}\label{cor: acl=dcl by RVexpansion}
Let $\ell$ be an integer and let $\Th_\cL(K)$ have property (T$\ell$, D) in some language $\cL$. Then there exists an expansion $\cL'$ of $\cL$ such that $\Th_{\cL'}(K)$ still has property (T$\ell$, D), and such that $\Th_{\cL'}(K)$ has $\acl=\dcl$.
\end{cor}

\begin{proof}
The language $\cL'$ can be obtained by an $\RV$-expansion by \cite[Prop.\,4.3.3]{CHR}, which preserves property (T$\ell$, D) by the previous theorem.
\end{proof}

Let $\cL'$ be an $\RV$-expansion of $\cL$. To move from $\cL'$-definable objects to $\cL$-definable ones, we will use the following lemma.

\begin{lem}
Let $\cL'$ be an $\RV$-expansion of $\cL$ and assume that $\Th_\cL(K)$ is 1-h-minimal. Let $\chi': K^n\to \RV^{k'}$ be an $\cL'$-definable decomposition. Then there exists an $\cL$-definable decomposition $\chi: K^n\to \RV^k$ such that every twisted box of $\chi$ is contained in a twisted box of $\chi'$.
\end{lem}

\begin{proof}
First use \cite[Lem.\,4.3.4]{CHR} to obtain an $\cL$-definable map $\psi: K^n\to \RV^{k''}$ such that every fibre of $\psi$ is contained in a twisted box of $\chi'$. Now use cell decomposition \cite[Thm.\,5.2.4, Add.\,1]{CHR} to find the desired decomposition $\chi$.
\end{proof}

\subsection{From (T$\ell$, D) to $\ell$-h-minimality}

We can now finish the proof of Theorem \ref{thm: equivalences}.

\begin{lem}\label{lem: Tell,D to hmixmin}
Assume that $\cT$ satisfies (T$\ell$, D). Then $\cT$ is $\ell$-h-minimal.
\end{lem}

\begin{proof}
We induct on $\ell$. Since the case $\ell=1$ is already proven, we may assume that $\ell>1$. Let $K$ be a model of $\cT$. By Corollary \ref{cor: acl=dcl by RVexpansion} there exists an $\RV$-expansion $\cL'\supset\cL$ of the language such that the resulting theory of $K$ in $\cL'$ has $\acl=\dcl$ and still satisfies property (T$\ell$, D). Moreover, $\Th_{\cL'}(K)$ is 1-h-minimal, by induction. We only use this, i.e.\ that $\Th_{\cL'}(K)$ is 1-h-minimal, satisfies (T$\ell$, D) and has $\acl=\dcl$.

Let $X$ be $\cL(A\cup \RV\cup A')$-definable, with $A\subset K$ and $A' = \{\xi_1, ..., \xi_\ell\}\subset \RV_\lambda$ for some $\lambda\in\Gamma_K^\times, \lambda\leq 1$. We want to $\lambda$-prepare $X$ with a finite $A$-definable set $C$. Consider $X$ as an $(A\cup\RV)$-definable family $(X_y)_y$ with $y\in K^\ell$ such that $X_y = X$ if $y\in Y := \prod_i \rv_\lambda^{-1}(\xi_i)$. For every $y\in K^\ell$ take a finite $\cL(A\cup\RV\cup\{y\})$-definable set $C_y\subset K$ such that $C_y$ 1-prepares $X_y$. By compactness, we can assume that $(C_y)_y$ is a $\cL(A\cup\RV)$-definable family with a parameter $y\in K^\ell$. Using $\acl=\dcl$, there exist finitely many $\cL'(A\cup \RV)$-definable functions
\[
c_1, ..., c_n: K^\ell\to K
\]
such that for any $y\in K^\ell$, $C_y$ is a subset of $\{c_1(y), ..., c_n(y)\}$. Define 
\[
\rv: K^\ell\to \RV^{\ell}: (y_1, ..., y_\ell)\mapsto (\rv(y_1), ..., \rv(y_\ell)),
\]
and for $i=1, \ldots, n$, let $Z_i = \{y\in K^\ell\mid c_i(y)\in X_y\}$. By property (T$\ell$,D) in $\cL'$ we can find an $\cL'(A\cup \RV)$-definable decomposition $\psi_i: K^\ell\to \RV^{m_i}$ such that $\rv$ is constant on the twisted boxes of $\psi_i$ and such that $Z_i$ is a union of twisted boxes of $\psi_i$. Use property (T$\ell$, D) to refine this to an $\cL'(A\cup\RV)$-definable decomposition $\chi_i: K^\ell\to \RV^{k_i}$ such that also $c_i$ has property (T$\ell$) on the twisted boxes of $\chi_i$. Let $C_i$ be a finite $\cL'(A\cup\RV)$-definable set 1-preparing $c_i(\chi_i^{-1}(\xi))$ for any $\xi\in \RV^{k_i}$. This uses family preparation under 0-h-minimality~\cite[Prop.\,2.6.2]{CHR}. Let $C=C_1\cup ... \cup C_n$, which is a finite $\cL'(A\cup\RV)$-definable set. 

We claim that $C$ $\lambda$-prepares $X$. So take $x\in X, x'\not\in X$. We want to find a $d\in C$ such that $|x-d|\lambda\leq |x-x'|$. Let $B_1$ be the smallest closed ball containing $x$ and $x'$. Fix a $y_0$ in $Y$. Since $C_{y_0}$ 1-prepares $X = X_{y_0}$, there exists an $i$ such that $c_i(y_0)\in B_1$. Let $F$ be the twisted box of $\chi_i$ containing $y_0$ and let $F'$ be the $\lambda$-twisted box containing $y_0$. Put $Y' = Y\cap F$ and note that $F'\subset Y$ since $\rv$ is constant on $F$. We claim that $c_i(F')\subsetneq B_1$. By property (T$\ell$) for $c_i$, $c_i(F')$ is either a singleton or an open ball. If it is a singleton, then so is $c_i(F)$. But then $c_i(F)$ is contained in a finite $\cL'(A\cup\RV)$-definable set, which already 1-prepares $X$. Indeed, the set of $\xi\in \RV^{k_i}$ for which $c_i$ is constant on $\chi_i^{-1}(\xi)$ is $\cL'(A\cup \RV)$-definable, and therefore the resulting map 
\[
\RV^{k_i}\to K: \xi\mapsto c_i(x), \text{ for any } x\in \chi_i^{-1}(\xi)
\]
has finite $\cL'(A\cup\RV)$-definable image by~\cite[Cor.\,2.6.7]{CHR}. So we can assume that $c_i(F')$ is an open ball. Both $c_i(F')$ and $B_1$ contain $c_i(y_0)$. The set $Z_i$ is partitioned by the twisted boxes of $\chi_i$, and so $c_i(Y')$ is either disjoint from $X$ (if $F\cap Z_i=\emptyset$), or contained in $X$ (if $F\subset Z_i$). But $B_1$ is neither contained, nor disjoint from $X$, from which we conclude that $c_i(F')\subsetneq B_1$. 

By property (T$\ell$) for $c_i$, $c_i(F')$ is an open ball of radius 
\[
\radop (c_i(F')) \leq \radcl (B_1) = |x-x'|.
\]
Since $C$ 1-prepares $c_i(F)$, there exists some $d\in C$ such that
\[
|c_i(y_0)-d| \leq \radop(c_i(F)).
\]
Using property (T$\ell$) for $c_i$ again, this yields that
\[
|c_i(y_0)-d|\lambda\leq \lambda\radop(c_i(F))  = \radop (c_i(F'))\leq \radcl B_1 = |x-x'|.
\]
Using that $c_i(y_0)\in B_1$, we conclude that
\[
|x-d|\lambda \leq \max\{|x-c_i(y_0)|\lambda, |c_i(y_0)-d|\lambda\} \leq |x-x'|.
\]
Finally, since $\cL'$ is obtained from $\cL$ by an $\RV$-expansion, \cite[Cor.\,4.1.17]{CHR} implies that there exists a finite $\cL(A\cup\RV)$-definable set $D$ containing $C$. To remove the $\RV$-parameters, we apply~\cite[Cor.\,2.6.10]{CHR} to obtain a finite $\cL(A)$-definable set $D'$ containing $D$. This set still $\lambda$-prepares $X$.
\end{proof}

\section{Hensel minimality in mixed characteristic}\label{sec: mixed char}

In this section we develop $\ell$-h-minimality in mixed characteristic. The main result is that $\ell$-h-minimality is preserved under coarsening of the valuation. 

\subsection{Coarsening the valuation}

Let $\cT$ be a theory of Henselian valued fields of characteristic zero (so possibly of mixed characteristic), in a language $\cL$ expanding the language of valued fields. Let $K$ be a model of $\cT$. We denote by $\cO_{K, \eqc}$ the smallest subring of $K$ containing $\cO_K$ and $\QQ$. This is again a valuation ring and we denote the corresponding valuation by $|\cdot|_\eqc: K\to \Gamma_{K,\eqc}$. This is the finest coarsening for which the resulting valued field is of equicharacteristic zero. In general, the coarsened valuation can be trivial (e.g.\ for $\QQ_p$). If the coarsening is non-trivial then we use the notation $\rv_\eqc$ and $\RV_\eqc$ to denote the leading term structure with respect to $|\cdot |_\eqc$. We denote by $\cL_{\eqc}$ the expansion of $\cL$ by a predicate for $\cO_{K, \eqc}$.

If $|\cdot|_c$ is any non-trivial coarsening of the valuation, then we use the notation $\cO_{K, c}$ for the coarsened valuation ring, as well as $\RV_c, \rv_c$ for the leading term structure. We use the notation $\cL_c$ for the expansion of $\cL$ by a predicate for $\cO_{K,c}$. We will use the notation $K_c$ to refer to $K$ equipped with the coarsened valuation $|\cdot |_c$.

\begin{thm}\label{thm: coarsening the valuation}
Assume that $\cT$ is $\ell$-h-minimal for some $\ell\geq 1$ and that all models have equicharacteristic zero. Let $K$ be a model of $\cT$ and let $|\cdot|_c$ be a non-trivial coarsening of the valuation. Then the theory of $K$ in $\cL_c$ with respect to the coarsened norm $|\cdot |_c$ is still $\ell$-h-minimal.
\end{thm}
\begin{proof}
By Theorem \ref{thm: equivalences}, property (T$\ell$, D) is equivalent to $\ell$-h-minimality. So we can use both in what follows.

Denote by $H\leq \Gamma_K^\times$ the convex subgroup by which we coarsen the valuation. Since the valuation is non-trivial, $H\neq \Gamma_K^\times$. If $H=1$, then there is nothing to prove. By Theorem \ref{thm: RVexpansions} we may expand the language $\cL$ by a predicate for the coarsened valuation ring $\cO_{K,c}$ and preserve property (T$\ell$, D) for $\Th_{\cL_c}(K)$ with respect to the valuation $|\cdot|$. Indeed, $\cO_{K,c}$ is a pullback of a predicate on $\RV$. We work in this larger language $\cL_c$. Define
\[
H' = \{\mu\in \Gamma_K^\times \mid \mu<\lambda,  \forall \lambda\in H\}.
\]
Note that $H'$ is $\emptyset$-definable in $\cL_c$ without maximal element and that 
\[
\cM_{K,c} = \bigcap_{\lambda\in H} B_{<\lambda}(0) = \bigcup_{\mu\in H'}B_{<\mu}(0).
\]
We verify the definition for $\ell$-h-minimality with respect to the coarsened valuation. So let $A\subset K, A'\subset \RV_c$ and $A''\subset \RV_{c,\lambda}, \#A''\leq \ell$ for some $\lambda\in \Gamma_{K, c}^\times, \lambda\leq 1$. Since we can freely add constants from $K$ and preserve $\ell$-h-minimality \cite[Lem.\,2.4.1]{CHR}, we may as well assume that $A = \emptyset$. Without loss of generality, $A'$ is finite, so we induct on $r = \#A'$. Let $X$ be $(A'\cup A'')$-definable. We wish to $\lambda$-prepare $X$.

Assume first that $r=0$. Take $\lambda_0\in \Gamma_K^\times, \lambda_0\leq 1$ reducing to $\lambda$ modulo $H$. Let $A'' = \{\xi_1, ..., \xi_\ell\}$. For every $\mu\in \lambda_0 H'$ there is a surjection $\RV_\mu \to \RV_\lambda$. So there is a $\emptyset$-definable family of sets $X(\zeta_1, ..., \zeta_\ell)\subset K$ with $\zeta_i\in \RV_\mu$ and $X(\zeta_1, ..., \zeta_\ell) = X$ if $(\zeta_1, ..., \zeta_\ell)$ maps to $(\xi_1, ..., \xi_\ell)$ under the natural surjection $\RV_\mu\to \RV_\lambda$. Moreover, this $\emptyset$-definable family is uniformly definable over the parameters $\zeta_i$ for all $\mu\in \lambda_0 H'$. Consider the $\lambda_0$-definable set
\[
W = \{(x, \zeta_1, ..., \zeta_\ell) \in K\times \RV_\bullet^\ell\mid \zeta_1, ..., \zeta_\ell\in \RV_\mu, \mu\in \lambda_0 H', x\in X(\zeta_1, ..., \zeta_\ell)\}.
\]
By uniform preparation under $\ell$-h-minimality \cite[Prop.\,2.6.2]{CHR}, there is a finite $\lambda_0$-definable set $C$ uniformly preparing this set $W$. By enlarging $C$ if necessary \cite[Cor.\,2.6.10]{CHR}, we may even assume that $C$ is $\emptyset$-definable. Let $B$ be a ball $\lambda$-next to $C$ and take $x, x'\in B$. Then 
\[
\forall d\in C: |x-x'|_c < \lambda|x-d|_c,
\]
or equivalently
\[
\forall d\in C\, \forall \nu\in H: |x-x'| < \lambda_0 \nu |x-d|.
\]
This means that for any $d\in C$, we have that $\frac{|x-x'|}{\lambda_0|x-d|}$ belongs to $H'$. Since $H'$ has no maximal element, there exists a $\mu\in \lambda_0 H'$ such that $|x-x'| < \mu|x-d|$ for any $d\in C$. But this means that $\rv_\mu(x-d) = \rv_\mu(x'-d)$ for any $d\in C$. Take $\zeta_1, ..., \zeta_\ell$ in $\RV_\mu$ mapping to $\xi_1, ..., \xi_\ell$ under $\RV_\mu\to \RV_\lambda$. Then $X(\zeta_1, ..., \zeta_\ell) = X$. By preparation of $W$, either $x$ and $x'$ are both in $X$ or they are both not in $X$. This proves the case $r=0$.

Now assume that $r>0$ and that the result is true up to $r-1$. We follow the strategy from \cite[Thm.\,2.9.1]{CHR}. Let $A' = A\cup \{\zeta\}$ and $\# A\leq r-1$. Consider $X$ as an $(A\cup A'')$-definable family $(X_y)_y$, $y\in K$ such that $X_y = X$ if $y$ is in $\rv_c^{-1}(\zeta) =:Y$. By adding constants from $K$ and induction, for every $y\in K$ there exists a finite $y$-definable set $(C_y)_y$ such that $X_y$ is $\lambda$-prepared by $C_y$. By compactness, we may without loss assume that $(C_y)_y$ is a $\emptyset$-definable family with a parameter $y\in K$. We apply Lemma~\ref{lem:finite.set.to.functions} to the family $(C_y)_y$ to obtain a $\emptyset$-definable family $(Y_\eta)_{\eta\in \RV^k}$ of subsets of $K$, together with a $\emptyset$-definable family of functions $g_\eta: Y_\eta\to K$ such that for each $x\in K$
\[
E_x = \bigsqcup_{\substack{\eta\in \RV^k \\ x\in Y_\eta}} \{g_\eta(x)\}.
\]
Using family preparation \cite[Prop.\,2.6.2]{CHR}, let $D'$ be a finite $\emptyset$-definable set 1-preparing every $Y_\eta$ and such that every $g_\eta$ satisfies the valuative Jacobian property \cite[Cor.\,3.2.7]{CHR} on balls $1$-next to $D'$. So every $g_\eta$ maps open balls 1-next to $D'$ to open balls and preserves radii of subballs compatibly. In particular, note that this also holds with respect to the coarsened valuation $|\cdot|_c$. By enlarging $D'$, we may assume that $D'$ also 1-prepares the sets $Z_\eta = \{y \in Y_\eta\mid g_\eta(y)\in X_y\}$, for every $\eta\in \RV^k$. Put $D = \{0\}\cup D'$, which is a finite $\emptyset$-definable set.

If $D\cap Y\neq \emptyset$ then the finite $\emptyset$-definable set $\cup_{y\in D}C_y$ will $\lambda$-prepare $X$. So assume that $Y\cap D = \emptyset$. The family of balls 1-next to $D$ forms a $\emptyset$-definable family parametrized by $\RV$-variables, by \cite[Lem.\,2.5.4]{CHR}. So we can find some $\emptyset$-definable finite set $C\subset K$ such that for each $\eta\in \RV^k$, $C$ 1-prepares every set of the form $g_\eta(B)$ where $B$ is an open ball 1-next to $D$ and contained in $Y_\eta$. Note that this also implies that $C$ $|1|_c$-prepares $g_\eta(B')$ for any ball $B'$ which is $|1|_c$-next to $D$ and contained in $Y_\eta$. Indeed, this follows from the valuative Jacobian property for $g_\eta$.

We claim that this $C$ is as desired. Let $x\in X, x'\not\in X$. We want to find $d\in C$ such that $|x-d|_c\lambda\leq |x-x'|_c$. Let $B_1'$ be the smallest closed ball containing both $x, x'$ with respect to the coarsened valuation. Let $B_1$ be the closed ball around $x$ of radius $\radclc (B_1')/\lambda$ and note that $B_1'\subset B_1$. Fix $y_0\in Y$. Since $C_{y_0}$ $\lambda$-prepares $X$, $C_{y_0}\cap B_1$ is non-empty. So there is some $\eta\in \RV^k$ such that $y_0\in Y_\eta$ and $g_\eta(y_0)\in B_1$. 

We claim that $g_\eta(Y)\subsetneq B_1$. Since $D$ 1-prepares $\{y\in Y_\eta\mid g_\eta(y)\in X_y\}$ and since $X_y = X$ for $y\in Y$, we see that $g_\eta(Y)$ is either contained in $X$ or disjoint from $X$. By our choice of $\eta$, $g_\eta(Y)\cap B_1$ is non-empty. But $B_1$ is neither contained in $X$, nor disjoint from $X$. Hence we conclude that $g_\eta(Y)$ is a proper subset of $B_1$. This implies that
\[
\radopc g_\eta(Y) \leq \radclc B_1 = |x-x'|_c / \lambda.
\]
By the discussion above, $C$ $|1|_c$-prepares $g_\eta(Y)$ and so there exists some $d\in C$ such that
\[
|g_\eta(y_0) - d|_c\lambda \leq \radopc (g_\eta(Y))\lambda \leq |x-x'|_c.
\]
Putting this together and using that $g_\eta(y_0)\in B_1$, we obtain that
\[
|x-d|_c\lambda\leq \max\{|x-g_\eta(y_0)|_c\lambda, |g_\eta(y_0)-d|_c\lambda\}\leq |x-x'|_c.
\]
This completes the proof.
\end{proof}

\subsection{Definitions and basic results}

Let $\cT$ be a theory of valued fields of characteristic zero in a language $\cL$ expanding the language of valued fields $\cL_\val = \{+, \cdot, \cO\}$. In this section we explicitly allow mixed characteristic as well.

\begin{defn}
Let $\ell\geq 0$ be an integer or $\omega$. We say that $\cT$ is $\ell$-\hc-minimal if the following holds. For any model $K$ of $\cT$ and every non-trivial equicharacteristic zero coarsening $|\cdot|_c$ of the norm on $K$, the theory of $K$ in $\cL_c$ is $\ell$-h-minimal with respect to $|\cdot|_c$.
\end{defn}

Let $\ell\geq 1$. Note that if all models of $\cT$ are of equicharacteristic zero, then in view of Theorem \ref{thm: coarsening the valuation} $\ell$-\hc-minimality is equivalent to $\ell$-h-minimality as defined in section \ref{sec:defn.equivalences}. In the next section we will prove that $\ell$-\hc-minimality implies $\ell$-\hmix-minimality. However, we don't know whether the converse is true as well for $\ell\geq 2$. For $\ell=1$, this is part of~\cite[Thm.\,2.2.7]{CHRV}. For this reason, we will not define a notion of $\ell$-h-minimality in mixed characteristic, and use $\ell$-\hc-minimality instead.

Theorem \ref{thm: coarsening the valuation} only applies if $\ell\geq 1$. If $\ell=0$ however, it is not obvious whether $0$-h-minimality and $0$-\hc-minimality are equivalent, even in equicharacteristic zero. This tameness notion lies between 0-h-minimality and 1-h-minimality. 

In fact, 0-\hc-minimality is equivalent to the following ``up to a power preparation''. Assume that all models of $\cT$ have equicharacteristic zero. Then $\cT$ is 0-\hc-minimal if and only if the following holds for every model $K$ of $\cT$. Let $\lambda$ be in $\Gamma_K^\times, \lambda\leq 1$, $A\subset K$, $A'\subset \RV_\lambda$, and let $X$ be $(A\cup \RV\cup A')$-definable. Then there exists an integer $N$ such that $X$ is $\lambda^N$-prepared by a finite $A$-definable set $C\subset K$. It would be interesting to see how much of the theory of Hensel minimality works under this weaker notion of 0-\hc-minimality, compared to 1-h-minimality.


\subsection{Transferring results from equicharacteristic zero.}


%

We recall how one can transfer results from the equicharacteristic zero setting to mixed characteristic, following~\cite{CHRV}. First we show that $\ell$-\hc-minimal theories are also $\ell$-\hmix-minimal.

\begin{cor}
Assume that $\cT$ is $\ell$-\hc-minimal, for some integer $\ell\geq 1$. Then $\cT$ is also $\ell$-\hmix-minimal.
\end{cor}

\begin{proof}
The argument is identical to~\cite[Cor.\,2.5.5]{CHRV}.
\end{proof}

We expect that $\ell$-\hc-minimality is in fact equivalent to $\ell$-\hmix-minimality. Then we could call it simply $\ell$-h-minimality. For $\ell=1$, this is proven in \cite[Thm.\,2.2.7]{CHRV}. By the usual compactness, this implies a family preparation result under $\ell$-\hc-minimality which was already stated in~\cite{CHRV}.

\begin{lem}[{\cite[Cor.\,2.3.8]{CHRV}}]\label{lem:uniform.prep.ellhmin}
Let $\cT$ be $\ell$-\hc-minimal for some positive integer $\ell\geq 1$, and let $K$ be a model of $\cT$. Let $\lambda$ be in $\Gamma_K^\times, \lambda\leq 1$, let $n,M$ be positive integers and let
\[
W\subset K\times \RV_M^n \times \RV_\lambda^\ell
\]
be $\emptyset$-definable. Then there exists a finite $\emptyset$-definable set $C\subset K$, and a positive integer $N$ such that for every ball $B$ which is $\lambda|N|$-next to $C$, the fibre $W_x\subset \RV_M\times \RV_\lambda^\ell$ does not depend on $x$ as $x$ runs over $B$.
\end{lem}

For use in the next section, we mention a little lemma about lifting maps $\chi: K^n\to \RV_N$ to $\chi: K^n\to \RV_M$ for some $|M|\leq |N|$.

\begin{lem}\label{lem:enlarging.integer}
Let $\cT$ be $\ell$-\hc-minimal for some positive integer $\ell\geq 1$, and let $K$ be a model of $\cT$. Let $\lambda$ be in $\Gamma_K^\times, \lambda\leq 1$, let $M_1, \ldots, M_r, M_1', \ldots, M_\ell'$ be positive integers, and let
\[
\chi: K^n\to \RV_{M_1}\times \ldots \times \RV_{M_r}\times \RV_{|M_1'|\lambda}\times \ldots \RV_{|M_\ell'|\lambda}
\]
be $\emptyset$-definable. Then there is a positive integer $M$ and a $\emptyset$-definable map $\chi': K^n\to \RV_M^r\times \RV_{|M|\lambda}^\ell$ such that $\chi = p\circ \chi'$, where 
\[
p: \RV_{M_1}\times \ldots \RV_{M_r}\times \RV_{|M_1'|\lambda}\times \ldots \RV_{|M_\ell'|\lambda}\to \RV_M^r\times \RV_{|M|\lambda}^\ell
\]
is the natural surjection.
\end{lem}
\begin{proof}
We induct on $n$. For $n=1$, take $M' = M_1\cdots M_rM_1'\cdots M_\ell'$ and let
\[
W = \{ (x,\xi)\in K\times \RV_{M'}^r\times \RV_{|M'|\lambda}^\ell \mid \chi(x) = p(\xi)\},
\]
where $p$ is the natural projection as above but with domain $\RV_{M'}^r\times \RV_{|M'|\lambda}^\ell$. Using Lemma~\ref{lem:uniform.prep.ellhmin} we obtain a finite $\emptyset$-definable set $C\subset K$ which uniformly $\lambda|N|$-prepares $W$. Now use~\cite[Lem.\,2.3.1(5)]{CHRV} to obtain the desired map $\chi'$ (this might also change $M'$).

For $n>1$, we may use compactness to find a $\emptyset$-definable family of maps $\chi'_x: K^{n-1}\to \RV_M^r\times \RV_{|M|\lambda}^\ell$ for $x\in K$ which satisfy the conclusion of the lemma for each map $\chi_x: K^{n-1} \to \prod_i \RV_{M_i} \times \prod_i \RV_{|M_i'|\lambda}$. But then the map $(x,y)\mapsto \chi'_x(y)$ is as desired.
\end{proof}

To transfer a cell decomposition from equichracteristic zero to mixed characteristic, see Lemma \ref{lem:celldecomp.Kecc.to.K}.

\section{Higher-dimensional preparation results}\label{sec:higher.dimension}

In this section we further develop higher dimensional geometry under $\ell$-h-minimality. The general philosophy is that $\ell$-h-minimality suffices to prepare objects which are at most $\ell$-dimensional in a controlled way. Note in particular, for $\omega$-h-minimality we obtain such results in all dimensions. Throughout, let $\cT$ be a theory of valued fields of characteristic zero, in a language $\cL$ expanding the language of valued fields. We state results also in mixed characteristic, but it should be noted that many results simplify in equicharacteristic zero.

\subsection{Uniform preparation and cell decompositions}

Here is a first result on preparation in higher dimensions, generalizing \cite[Prop.\,2.6.2]{CHR} and \cite[Prop.\,2.3.2]{CHRV}.

\begin{def-prop}[Preparing higher dimensional families]\label{prop:uniform}
Assume that $\Th(K)$ is $(L+\ell)$-\hc-minimal, for some integers $L\ge 0$ and $\ell\ge 0$. Let $N>0$ be an integer and $\lambda \leq 1$ in $\Gamma_K^\times$. For any integer $k\ge 0$ and any $\emptyset$-definable set
$$
W\subset K^{L+1}\times \RV_{N}^k\times \RV_{\lambda}^\ell,
$$
there exists an integer $M$ and a $\emptyset$-definable function $\chi: K^{L+1}\to \RV_{M}^m\times \RV_{|M|\lambda}^{L+1}$ for some $m\ge 0, M>0$, such that for every fiber $F = \chi^{-1}(\eta)$ with $\eta\in \RV_{M}^m\times \RV_{|M|\lambda}^{L+1}$, the set $W_{x,\lambda} := \{\xi \in \RV_N^k\times \RV_{\lambda}^\ell \mid (x,\xi)\in W\}$ does not depends on $x$ when $x$ runs over $F$. In other words, for all $x,x'\in F$, one has $W_{x,\lambda} = W_{x',\lambda}$. We say that $\chi$ \emph{uniformly prepares} $W$.
\end{def-prop}

\begin{proof}
We induct on $L$. The base case $L=0$ follows from Lemma \ref{lem:uniform.prep.ellhmin} combined with \cite[Lem.\,2.3.1(5)]{CHRV}. So assume that $L>0$. Consider $W$ as a $\emptyset$-definable family $(W_y)_y$ with 
\[
W_y\subset K \times \RV_N^k \times \RV_\lambda^{\ell}
\]
where $y$ runs over $K^L$. By the base case and compactness, there exists integers $M', m'$ and a $\emptyset$-definable family $(\chi_y)_{y\in K^L}$ of functions 
\[
\chi_y: K\to \RV_{M'}^{m'}\times \RV_{|M'|\lambda}
\]
preparing the set $W_y$, i.e.\ any fibre of $\chi_y$ is either contained in or disjoint from $W_{y, \xi}\subset K$ for any $\xi\in \RV_N^k\times \RV_\lambda^\ell$. Consider the set 
\[
V = \{(y, \eta, \xi)\in K^L\times \RV^{m'}_{M'}\times \RV_{|M'|\lambda}\times \RV_N^k\times \RV_\lambda^\ell \mid \chi^{-1}_y(\eta)\subset W_{y, \xi}\}.
\]
By induction on $L$ there exists a $\emptyset$-definable function $\chi': K^L\to \RV^{m''}_{M''}\times \RV_{|M''|\lambda}^L$ which uniformly $\lambda$-prepares $V$. By Lemma \ref{lem:enlarging.integer} we may assume that $M'=M''=M$. Define
\[
\chi: K^{L+1}\to \RV_M^{m''}\times \RV_{|M|\lambda}^{L} \times \RV_M^{m'}\times \RV_{|M|\lambda}: (y,x)\mapsto (\chi'(y), \chi_y(x)).
\]
We claim that this map is as desired. Let $(y,x)$ and $(y',x')$ be in the same fibre of $\chi$ and put $\eta=\chi_y(x)=\chi_{y'}(x')$. Suppose that $(y,x,\xi)\in W$ for some $\xi\in \RV_N^k\times \RV_\lambda^\ell$. Then $(y,\eta, \xi)$ is in $V$. Since $\chi'(y)=\chi'(y')$ also $(y',\eta,\xi)$ is in $V$. This in turn implies that $(y',x',\xi)\in W$. We are done by symmetry.
\end{proof}

One may restate the above result using cell decompositions, as long as $\acl=\dcl$. For $\lambda\in \Gamma_K^\times, \lambda\leq 1$, define $0\cdot \RV_\lambda^\times  = \{0\}$ and $1\cdot\RV_\lambda^\times = \RV_\lambda^\times$.

\begin{defn}[{\cite[Def.\,5.2.2]{CHR}, \cite[Def.\,3.3.1]{CHRV}}]
Let $A\subset K\cup \RV$ be a parameter set, $N$ a positive integer and take $j_1, \ldots, j_n$ in $\{0,1\}$. Then an \emph{$A$-definable cell $X$ of type $(j_1, \ldots, j_n)$ of depth $N$} is a set of the form
\[
X = \{x\in K^n\mid \left( \rv_N(x - c_i(\pi_{<i}(x))\right)_{i=1}^n \in R\},
\]
where the functions $c_i: \pi_{<i}(X)\to K$ are $A$-definable and $R\subset \prod_i j_i\cdot \RV_N^\times$ (which is automatically $A$-definable). The $c_i$ are the \emph{cell centers of $X$}. For $\lambda \leq |N|$ in $\Gamma_K^\times$, a $\lambda$-twisted box of $X$ is a set of the form
\[
\{x\in K\mid \left( \rv_\lambda(x - c_i(\pi_{<i}(x))\right)_{i=1}^n = (\xi_i)_{i=1}^n\},
\]
where $\xi_i\in \RV_\lambda$ lies above an element of $R$ under the natural map $\RV_\lambda^n\to \RV_N^n$.

An $A$-definable \emph{cell decomposition $\cA$ of $K^n$ of depth $N$} is a finite partition of $K^n$ into $A$-definable cells of depth $N$.
\end{defn}

\begin{cor}
Assume that $\Th(K)$ is $(L+\ell)$-\hc-minimal and has $\acl=\dcl$, for some integers $L\ge 0$ and $\ell\ge 0$. Let $\lambda \leq 1$ in $\Gamma_K^\times$. For any $\emptyset$-definable set
\[
W\subset K^{L+1}\times \RV_M^k\times \RV_{\lambda}^\ell,
\]
there exists a cell decomposition $\cA$ of $K^{L+1}$ of depth $N$ (for some positive integer $N$) such that for every $|N|\lambda$-twisted box $F$ of $\cA$, the set $W_x$ is independent of $x$ as $x$ runs over $F$.
\end{cor}
\begin{proof}
This is similar to the previous proposition, see also \cite[Thm.\,5.2.4 Add.\,1]{CHR}.
\end{proof}

In transferring results from equicharacteristic zero to mixed characteristic, the following lemma is frequently useful. 

Let $\cA$ be a cell decomposition of $K$ of depth $N$. By a $|1|_\eqc$-twisted box of $\cA$ we mean the following. Let $(c_1, \ldots, c_n)$ be the cell center of some twisted box of $\cA$ of the form
\[
\{x\in K^n\mid (\rv_{N}(x_i-c_i(\pi_{<i}(x))))_{i=1}^n = \xi\},
\]
for some $\xi\in \RV_{N}^n$. Then a $|1|_\eqc$-twisted box of $\cA$ is a set of the form
\[
\{x\in K^n\mid (\rv_{\eqc}(x_i-c_i(\pi_{<i}(x))))_{i=1}^n = \xi'\},
\]
for any $\xi'\in \RV_\eqc^n$ reducing to $\xi$ under the natural projection $\RV_\eqc\to \RV_{N}$.

\begin{lem}\label{lem:celldecomp.Kecc.to.K}
Assume $\Th_\cL(K)$ is 1-\hc-minimal and that $\acl=\dcl$. Let $\cA'$ be an $\cL_\eqc$-definable cell decomposition of $K^n_\eqc$. Then there exists an integer $N$, and an $\cL$-definable cell decomposition $\cA$ of $K^n$ of depth $N$ such that every $|1|_\eqc$-twisted box of $\cA$ is contained in a twisted box of $\cA'$.
\end{lem}

\begin{proof}
We can do this one cell at a time, so let $A$ be an $\cL_\eqc$-definable cell with cell center $(c_1, \ldots, c_n)$. By \cite[Cor.\,4.1.17]{CHR} (which holds in our context by \cite[Prop.\,2.6.4]{CHRV}), we have that $c_i(\pi_{<i}(x))$ is always $\cL(\pi_{<i}(x))$-definable. Let $C_i(\pi_{<i}(x))$ be an $\cL$-definable family of finite sets containing $c_i(\pi_{<i}(x))$, with parameter $\pi_{<i}(x)\in K^{i-1}$. Using $\acl=\dcl$ in $\cL$, we can find several $\cL$-definable cells of depth $|N|$, whose cell center is among the $C_i$. This will be the desired cell decomposition.
\end{proof}

\subsection{Compatible domain and image preparation}

Let $f: K\to K$ be a $\emptyset$-definable function. In \cite[Prop.\,2.8.6]{CHR} it is proven that under 1-h-minimality one may compatibly prepare the domain and the image of $f$. In more detail, given a finite $\emptyset$-definable set $C_0$, there exist finite $\emptyset$-definable sets $C\supset C_0$, and $D$ such that for every ball $B$ 1-next to $C$, $f(C)$ is either a ball 1-next to $D$ or contained in $D$. Using cell decompositions, one may even state such a result for finitely many maps $f_i: K\to K$ at once. We extend this result to higher dimensions, subject to $\ell$-h-minimality. If we only assume $1$-h-minimality, then we prove a weaker version of this result. We will assume that $\acl=\dcl$, which is no real loss in view of Corollary \ref{cor: acl=dcl by RVexpansion}. 

\begin{thm}\label{thm.f.ell.III}
Suppose that $K$ is of equicharacteristic zero, and carries an $\ell$-h-minimal structure for some integer $\ell>0$ with $\acl=\dcl$. Let $f_1, ..., f_r: K^\ell\to K$ be $\emptyset$-definable, and suppose we are given a cell decomposition $\cA'$ of $K^\ell$. Then there exists a $\emptyset$-definable cell decomposition of $K^\ell$ into cells $A_i$ refining $\cA'$ and $\emptyset$-definable points $c_{i,1}, ..., c_{i,r}\in K$ such that for every $i,j$, every $\lambda\leq 1$ in $\Gamma_K^\times$ and any $\lambda$-twisted box $F$ of $A_i$, $f_j(F)$ is either a ball $\lambda$-next to $c_{i,j}$, or $f_j(F)$ is equal to $c_{i,j}$.

If we only assume $1$-h-minimality then such a cell decomposition still exists, but the conclusion holds for $\lambda=1$ only. 
\end{thm}

Note that this result is a stronger version of property (T$\ell$), and the proof is again by induction on the dimension.

\begin{proof}
We induct on $\ell$, where the base case $\ell=1$ is known by \cite[Thm.\,5.2.4 Add.\,3]{CHR}.

Denote by $\cA'(x)$ the cell decomposition on $K$ induced by $\cA'$ when the first $\ell-1$ coordinates are equal to $x\in K^{\ell-1}$. Using the base case of induction and \cite[Thm.\,5.2.4 Add.\,1, Cor.\,3.1.6]{CHR}, take a cell decomposition $\cA(x)$ of $K$ with cells $A_i(x)$ refining $\cA'(x)$ and singletons $e_{ij}(x)$ such that the conclusion of the theorem holds for the $f_{j,x}$ with respect to the cells $A_i(x)$ and the singletons $e_{ij}(x)$, and such that moreover every $f_{j,x}$ has the valuative Jacobian property on twisted boxes in the cell decomposition $\cA(x)$. In more detail, for every $x\in K^{\ell-1}$ and every $i,j$, $f_{j,x}$ maps every $\lambda$-twisted box of the cell $A_i(x)$ to a ball $\lambda$-next to $e_{ij}(x)$. Denote by $c_i(x)$ the cell center of the cell $A_i(x)$. By compactness, we can assume that for every $i,j$ the functions
\begin{align*}
e_{ij} &: K^{\ell-1}\to K: x\mapsto e_{ij}(x) \\
c_i &: K^{\ell-1}\to K: x\mapsto c_i(x)
\end{align*}
are $\emptyset$-definable.

Fix $x\in K^{\ell-1}$ and fix $i, j$. Let $\zeta$ be in $\RV_\lambda$ such that $c_i(x) + \rv_\lambda^{-1}(\zeta)$ is contained in the cell $A_i(x)$. Then for $y\in c_i(x) + \rv_\lambda^{-1}(\zeta)$, the value
\[
\rv_\lambda(f_j(x,y) - e_{ij}(x))
\]
is independent of $y$, and depends only on $\zeta$. We thus obtain $\emptyset$-definable maps
\[
H_{ij}: K^{\ell-1}\times \RV_\bullet \to \RV_\bullet
\]
mapping $(x, \zeta)\in K^{\ell-1}\times \RV_\lambda$ to $\rv_\lambda(f_j(x,y) - e_{ij}(x))$, for $y \in c_i(x) + \rv_\lambda^{-1}(\zeta)$. Apply Lemma \ref{prop:uniform} to obtain a cell decomposition $\cB'$ of $K^{\ell-1}$ uniformly preparing all the $H_{ij}$ and refining $\cA'$, this is where we use $\ell$-h-minimality. By induction, we can further refine this cell decomposition $\cB'$ to obtain a cell decomposition $\cB$ with cells $B_k$ of $K^{\ell-1}$ and $\emptyset$-definable points $d_{ijk}$ such that for every $\lambda\in \Gamma_K^\times, \lambda\leq 1$, every $\lambda$-twisted box $F$ contained in the cell $B_k$, $e_{ij}(F)$ is either equal to $d_{ijk}$, or a ball $\lambda$-next to $d_{ijk}$. Let $\cA$ be the cell decomposition on $K$ induced by $\cB$ and $\cA(x)$. Let $F$ be a twisted box of this cell decomposition, then $\pi_{<\ell}(F) = B$ is a twisted box of the cell decomposition $\cB$, say contained in the cell $B_k$. There is a $\zeta\in \RV$ such that
\[
F = \bigcup_{x\in B} \{x\} \times (c_i(x) + \rv^{-1}(\zeta)),
\]
for some $i$. Now $f_j(F)$ is equal to 
\begin{align*}
f_j(F) &= \bigcup_{x\in B} f_{j,x}\left( c_i(x) + \rv^{-1}(\zeta)\right) \\
	&= \bigcup_{x\in B} e_{ij}(x) + \rv^{-1}\left( H_{ij}(x, \zeta)\right) \\
	&= e_{ij}(B) + \rv^{-1}(\xi) \\
	&= d_{ijk} + \rv^{-1}(\xi') + \rv^{-1}(\xi),
\end{align*}
where $\xi'\in \RV$, and $\xi = H_{ij}(x,\zeta)$ for any $x\in B$. By preparation of $H_{ij}$, this is indeed independent of $x$, for $x\in B$. Now if $0$ is not in $\rv^{-1}(\xi') + \rv^{-1}(\xi)$ then this is a ball $1$-next to $d_{ijk}$. While if $\xi=\xi'=0$ then $f_j(F) = \{d_{ijk}\}$. So the only case in which the conclusion doesn't hold for $F$, is if $\xi = -\xi'\neq 0$. For $x\in B$, the set $f_{j,x}^{-1}(d_{ijk})\cap \pi_\ell(F)$ is a singleton, by the Jacobian property, call it $d_{ijk'}$. Moreover, balls $\lambda$-next to $d_{ijk}'$ are mapped to balls $\lambda$-next to $d_{ijk}$, by the Jacobian property for $f_{j,x}$. Now refine the cell decomposition further for these twisted boxes, by taking the set $d_{ijk}'$ as cell centers. The result will be the cell decomposition with the desired property.

Only assuming $1$-h-minimality, the above proof goes through by replacing every $\lambda$ by $1$ (and $\RV_\lambda$ by $\RV$).
\end{proof}

\subsection{Supremum-Jacobian-property}

In this section we generalize the supremum-Jacobian property from \cite[Thm.\, 5.4.10]{CHR} and \cite[Thm.\, 3.3.6]{CHRV} and obtain a uniform version involving general $\lambda\leq 1$ rather that just $\lambda=1$. Such a generalization is similar to the difference between the valuative Jacobian property \cite[Cor.\,3.1.6]{CHR} and the full Jacobian property \cite[Cor.\,3.2.7]{CHR} in dimension $1$. As above, we prove such a result in $\ell$ dimensions under $\ell$-h-minimality. Recall that on product spaces $K^n$ we use the maximum norm.

\begin{defn}\label{def: higher dimensional RVlambda}
Let $\lambda\leq 1$ be in $\Gamma_K^\times$. We define $\RV_\lambda^{(n)}$ as the quotient $K^n/\sim$ where $x\sim x'$ if and only if $x=x'$ or $|x-x'| < \lambda |x|$. We will write $\rv_\lambda^{(n)}$ for the natural quotient map $K^n\to \RV_\lambda^{(n)}$.
\end{defn}

For $\lambda, \lambda'\in \Gamma_K$, we use the notation $\lambda\ltz \lambda'$ to mean that either $\lambda < \lambda'$ or $\lambda = \lambda' = 0$. 

Let $U\subset K^m$ be open and let $f: U\to K^n$ be a map. As usual, we say that $f$ is $C^0$ if it is continuous, and $C^k$ if there exists a $C^{k-1}$-map $J: U\to K^{n\times m}$ such that for every $u\in U$ we have that
\[
\lim_{x\to u} \frac{|f(x)-f(u)-J(u)\cdot (x-u)|}{|x-u|} = 0.
\]
In the h-minimal setting, every definable function $f: K^n\to K$ is almost everywhere $C^k$, in the sense that the set of $u\in K^n$ for which $f$ is not $C^k$ in any neighbourhood of $u$ is of dimension $<n$. Indeed, this follows from~\cite[Thm.\,5.1.5]{CHR} and cell decomposition.

\begin{defn}[Supremum Jacobian property]\label{defn:sup-prep.III}
Let $\lambda\le 1$ in $\Gamma_K^\times$ and an integer $n\ge 0$ be given. For $X \subset K^n$ open and $f\colon X \to K$ a function, we say that $f$ has the \emph{$\lambda$-supremum Jacobian Property} (\emph{$\lambda$-sup-Jac-prop} for short) on $X$ if
$f$ is $C^1$ on $X$,  
\[\label{eq:rvn.cst.III}
\rv_\lambda(\partial f/\partial x_i) \mbox{ is constant on $X$ for each } i=1,\ldots,n,
\]
and, for every $x$ and $y$ in $X$ we have:
\begin{equation}\label{eq:lambdasupJacprop}
|f(x) - f(y) -  \grad f(y)\cdot(x - y) | \ltz  \lambda\cdot |\grad f(y) |\cdot |x-y|.
\end{equation}
\end{defn}

In the definition, as usual, one considers $\grad f(y)$ as a matrix with a single row, which is multiplied with the column vector $x - y$. Note also that the left hand side of Equation~\eqref{eq:lambdasupJacprop} depends only on the value $\rv_\lambda^{(n)}(\grad f(y))$, a fact which we will use in the proof. The main result of this section is the following.

\begin{thm}[Supremum-Jacobian-preparation]\label{thm:lambdasupJacprop}
Suppose that $K$ is of equicharacteristic zero and that $\Th(K)$ is $\ell$-h-minimal for some integer $\ell>0$. For every $\emptyset$-definable function $f: K^\ell \to K$, there exists a $\emptyset$-definable decomposition $\chi: K^\ell\to \RV^k$ (for some integer $k$) such that for every $\lambda\leq 1$ in $\Gamma_K^\times$ and for every $\lambda$-twisted box $F$ of $\chi$, if $F$ is $\ell$-dimensional then $f$ restricted to $F$ has the $\lambda$-supremum Jacobian property.
%
\end{thm}

Instead of finding a decomposition, we will fix $\lambda$ and find a map $\chi: K^\ell\to \RV^k\times \RV_\lambda^\ell$ such that $f$ has the $\lambda$-supremum Jacobian property on the fibres of $\chi$. We introduce the following terminology. Let $f: K^n\to K$ be a function. We say that a map $\chi: K^n\to \RV^k\times \RV_\lambda^n$ \emph{$\lambda$-sup-prepares $f$} if each $n$-dimensional fibre $F$ of $\chi$ is open and $f$ has the $\lambda$-sup-Jac-prop on $F$. To obtain a decomposition as in the theorem from such a map $\chi$, one uses family preparation \cite[Prop.\,2.6.2]{CHR}.

We will first prove the theorem in equicharacteristic zero. Using compactness, we will then prove this result also in mixed characteristic. So for now, assume that $K$ is of equicharacteristic zero. By induction on $\ell$ we can assume that the result already holds up to $\ell-1$. We first show that one may assume that $\acl=\dcl$.

\begin{lem}
It suffices to prove Theorem \ref{thm:lambdasupJacprop} under the assumption that $\acl=\dcl$ in $K$.
\end{lem}
\begin{proof}
This is similar to \cite[Claim.\,5.4.13]{CHR}.
\end{proof}

So for the rest of the proof, we will assume that $\acl=\dcl$ in $K$.

\begin{lem}\label{lem:preservation.supjacprop}
Let $X\subset K^m$, $Y\subset K^n$ be open and let $\alpha: X\to Y$ be $C^1$. Suppose that $\rv_\lambda^{(n\times m)}(\Jac \alpha)$ is constant on $X$ and that for $x_1, x_2\in X$ we have
\[
|\alpha(x_2)-\alpha(x_1)| = |\Jac \alpha||x_2-x_1|
\]
and
\[
\rv_\lambda^{(n)} (\alpha(x_2)-\alpha(x_1)) = \rv_\lambda^{(n)}((\Jac(\alpha)(x_1)\cdot(x_2-x_1)).
\]
Finally, suppose that $f: Y\to K$ is $C^1$ such that $f\circ \alpha$ has the $\lambda$-sup-Jac-prop on $X$. Then $f$ satisfies Equation~\eqref{eq:lambdasupJacprop}.
\end{lem}
\begin{proof}
Similar to \cite[Lem.\,5.4.9]{CHR}.
\end{proof}

\begin{lem}\label{lem:claim.supjacprop}
Let $Z\subset K^\ell$ be $(\ell-1)$-dimensional $\emptyset$-definable and let $f$ be $C^1$ on some open neighbourhood of $Z$. Then there is a $\emptyset$-definable map $\chi: Z\to \RV^k\times \RV_\lambda^\ell$ such that if $F$ is an $(\ell-1)$-dimensional fibre of $\chi$ and $x, x_0\in F$ then
\[
|f(x) - f(y) -  \grad f(y)\cdot(x - y) | <  \lambda\cdot |\grad f(y) |\cdot |x-y|.
\]
\end{lem}
\begin{proof}
By cell decomposition \cite[Thm.\,5.2.4 Add.\,1,5]{CHR}, we can find a cell decomposition $\cA$ of $Z$ such that on each $\lambda$-twisted box of $\cA$, $\rv^{(\ell)}_\lambda(\grad f)$ is constant and that after a coordinate permutation every cell is of type $(1, ..., 1, 0, ..., 0)$ with $1$-Lipschitz centers. Fix a cell $A$ of $\cA$. We can assume that $A$ is a cell of type $(1, ..., 1, 0)$ with $1$-Lipschitz center $c$. Let $\hat{A}$ be the projection of $A$ onto the first $(\ell-1)$ coordinates. Then $A$ is the graph of the $1$-Lipschitz function $c: \hat{A}\to K$, where $c$ is the last center coordinate. Let $\alpha: \hat{A}\to A: x\mapsto (x, c(x))$ and apply inductively the $\lambda$-sup-Jac-prop for $c$ and $f\circ \alpha$ to obtain cell decompositions $\chi_1, \chi_2$. Let $\cB$ be a cell decomposition refining both $\chi_1$ $\chi_2$. Since $c$ is $1$-Lipschitz, $\alpha$ satisfies the assumptions of Lemma \ref{lem:preservation.supjacprop} on every $\lambda$-twisted box of $\cB$. Thus the fact that $f\circ \alpha$ satisfies Equation~\eqref{eq:lambdasupJacprop} on every $\lambda$-twisted box of $\cB$ yields that also $f$ satisfies Equation~\eqref{eq:lambdasupJacprop} on the corresponding cell decomposition of $A$. This argument works for every cell $A$ in $\cA$ and so we can combine them into one cell decomposition such that the desired result holds on every $\lambda$-twisted box of the resulting cell decomposition.
\end{proof}

\begin{proof}[Proof of Theorem \ref{thm:lambdasupJacprop}]
We are given $f: X\to K$ for $X\subset K^\ell$ and we wish to $\lambda$-sup-prepare $f$. Write elements of $K^\ell$ as $(x,y)$ with $x\in K^{\ell-1}$ and $y\in K$. Note that we can freely restrict $f$ to the cells of a cell decomposition if desired, since we can refine cell decompositions. So by cell decomposition and the fact that the set where $f$ is $C^1$ is definable, we can assume that $X$ is a cell, $f$ is $C^1$ and that $\rv_\lambda^{(\ell)}(\grad f)$ is constant on $\lambda$-twisted boxes of $X$. Moreover, applying~\cite[Thm.\,5.2.4 Add.\,1,5]{CHR} gives after a coordinate permutation that $X$ is a $(1, ..., 1)$-cell with $1$-Lipschitz centers. Let $c$ be the last center coordinate. (We can disregard lower dimensional cells.) 

For $x\in K^{\ell-1}$ let $\chi_x: K\to \RV^k\times \RV_\lambda$ be $\emptyset$-definable such that for every coordinate permutation $\sigma: K^\ell\to K^\ell$ the map $y\mapsto f(\sigma(x, y))$ has the $\lambda$-sup-Jac-prop on the fibres of $\chi_x$. Here we have used induction on $\ell$. By compactness we may assume that the family $(\chi_x)_x$ is $\emptyset$-definable.

Denote by $\hat{X}$ the projection of $X$ onto the first $\ell-1$ coordinates, so that $c$ is defined on $\hat{X}$. By induction, we can find a map $\chi_1: K^{\ell-1}\to \RV^k\times \RV_\lambda^{\ell-1}$ such that $c$ satisfies the $\lambda$-sup-Jac-prop on fibres of $\chi_1$. Any $\lambda$-twisted box of $X$ is of the form
\[
F = \{(x,y)\in \overline{X}\times K \mid \rv_\lambda(y-c(x)) = \rho\},
\]
for some $\rho\in \RV_\lambda$ and some $\lambda$-twisted box $\overline{X}\subset \hat{X}$. Consider the injective map
\[
\beta: X \to K^\ell: (x,y)\mapsto (x, y-c(x))
\]
and note that for fixed $\rho\in \RV_\lambda$, $\beta$ maps the $\lambda$-twisted box $F$ bijectively onto $\overline{X}\times \rv_\lambda^{-1}(\rho)$. On such a $\lambda$-twisted box, the inverse $\alpha = \beta^{-1}$ satisfies the requirements of Lemma \ref{lem:preservation.supjacprop}. Thus it suffices to $\lambda$-sup-prepare $f\circ \alpha$ on every $\lambda$-twisted box.

We have thus reduced to the following situation. The set $X$ is a cell, and every $\lambda$-twisted box is of the form $\overline{X}\times \rv_\lambda^{-1}(\rho)$ for some $\lambda$-twisted box $\overline{X}$ in $\hat{X}$ and some $\rho\in \RV_\lambda$, $f$ is $C^1$, $\rv_\lambda^{(\ell)}(\grad (f))$ is constant on $\lambda$-twisted boxes in $X$ and for every fixed $x$, $f(x, \cdot)$ has the $\lambda$-sup-Jac-prop when restricted to a $\lambda$-twisted box.

By induction and compactness, we find a map $\chi_2: X\to \RV^k\times \RV_\lambda^{\ell-1}$ such that for every fixed $y\in K$ the map $f(\cdot, y)$ is $\lambda$-sup-prepared by $\chi_2(\cdot, y)$. Pick a cell decomposition of $X$ with continuous centers such that $\chi_2$ is constant on any $\lambda$-twisted box of the cell decomposition. Suppose there are $n$ cells $A_1, ..., A_n$ in this cell decomposition. Fix a cell $A_i$ and let $\hat{A}_i$ be the projection onto the first $\ell-1$ coordinates. Let $c_i: \hat{A}_i\to K$ be the last center component.

Let $Z_i$ be the intersection of the graph of $c_i$ with $X$ and note that this is of dimension at most $\ell-1$. Apply Lemma \ref{lem:claim.supjacprop} to $Z_i$ and $f$, which yields a map $\chi'_i: Z_i\to \RV^k\times \RV_\lambda^\ell$ such that on every $(\ell-1)$-dimensional fibre of $\chi'_i$, $f$ satisfies Equation~\eqref{eq:lambdasupJacprop}. Extend every $\chi'_i$ by zero on $X$ and consider the map
\[
X\to \RV^k\times \RV_\lambda^{\ell-1} \times (\RV^k\times \RV_\lambda^{\ell})^n: (x,y)\mapsto (\chi_2(x,y), (\chi'_i(x, c_i(x))_i).
\]
Let $\cA$ be a cell decomposition refining this map. Denote the corresponding map $\chi: K^\ell\to \RV^k\times \RV_\lambda^\ell$ whose fibres are the $\lambda$-twisted boxes of the cell decomposition. Fix $j\in \{1,..., n\}$. If $(x_1, y_1)$ and $(x_2, y_2)$ are in the same $\ell$-dimensional fibre of $\chi$, and if $(x_1, c_j(x_1)), (x_2, c_j(x_2))$ are both in $X$ then they satisfy Equation~\eqref{eq:lambdasupJacprop}.

Let $F$ be a $\lambda$-twisted box of $\cA$, i.e.\ a fibre of $\chi$. Then $F$ is still of the form
\[
F = \{(x,y)\in K^\ell\mid x\in \overline{F}, \rv_\lambda(y-c_j(x)) = \xi\}
\]
for some $j\in \{1, ..., n\}$ and some $\xi\in \RV_\lambda$. Here $\overline{F} = \pi_{<\ell}(F)$. We claim that $f$ has the $\lambda$-sup-Jac-prop on $F$. We already know that $f$ is $C^1$ on $F$ and that $\rv_\lambda^{(\ell)}(\grad f)$ is constant on $F$. So fix $(x_1, y_1), (x_2, y_2)\in F$. Then certainly $\rv_\lambda y_1 = \rv_\lambda y_2$ by our construction. If $(x_3, y_3)\in F$ is chosen so that $\rv_\lambda y_3 = \rv_\lambda y_2$ and such that 
\[
|(x_1, y_1) - (x_3, y_3)| \leq |(x_1, y_1) - (x_2, y_2)|
\]
then a simple computation shows that if Equation~\eqref{eq:lambdasupJacprop} holds for $(x_1, y_1), (x_3, y_3)$ and for $(x_3, y_3), (x_2, y_2)$ then it also holds for $(x_1, y_1), (x_2, y_2)$. Thus in proving that Equation~\eqref{eq:lambdasupJacprop} holds for $(x_1, y_1), (x_2, y_2)$ we can jump to several intermediate points in $X$. We reason as follows.
\begin{enumerate}
\item If $|c_j(x_1)-y_1| > |y_2-y_1|$ then we have that $(x_1, y_2)\in F$. Hence we can jump from $(x_1, y_1)$ to $(x_1, y_2)$ since $f(x_1, \cdot)$ is $\lambda$-sup-prepared by $\chi$ by our construction above. Then we can jump from $(x_1, y_2)$ to $(x_2, y_2)$ since we also ensure that $f(\cdot, y_2)$ is $\lambda$-sup-prepared by $\chi$. 
\item If $|c_j(x_2)-y_2| > |y_2-y_1|$ then we reason similarly as in case 1.
\item So assume that $|c_j(x_i)-y_i| \leq |y_2-y_1|$ for $i=1, 2$. Then since $\rv_\lambda y_1 = \rv_\lambda y_2$, also $\rv_\lambda c_j(x_i) = \rv_\lambda y_1$ for $i=1,2$. Moreover, for $i=1,2$ we have that
\[
|(x_i, y_i) - (x_i, c_j(x_i))| \leq |(x_1, y_1) - (x_2, y_2)|.
\]
Hence by our construction we can jump from $(x_1, y_1)$ to $(x_1, c_j(x_1))$ since we have ensured that $f(x_1, \cdot)$ is $\lambda$-sup-prepared by $\chi$. We can then jump from $(x_1, c_j(x_1))$ to $(x_2, c_j(x_2))$ by the discussion above. Finally, we can jump from $(x_2, c_j(x_2))$ to $(x_2, y_2)$ by the same reason as before.
\end{enumerate}
This finishes the proof.
\end{proof}

The mixed characteristic case is deduced by compactness and coarsening.

\begin{thm}
Let $\cT$ be $\ell$-\hc-minimal for some $\ell>0$ and let $K$ be a model of $\cT$. Assume that $\acl=\dcl$ in $\cT$ and let $f: K^\ell\to K$ be $\emptyset$-definable. Then there exists a $\emptyset$-definable cell decomposition $\chi: K^\ell\to \RV_N^k$ of depth $N$ (for some integers $N, k$) such that for every $\lambda\leq |N|$ and for every $\lambda$-twisted box $F$ of $\chi$, if $F$ is $\ell$-dimensional then $f$ restricted to $F$ has the $|M|\lambda$-supremum Jacobian property (for some positive integer $M$).
\end{thm}

\begin{proof}
We may assume that $K$ is $\aleph_0$-saturated. For $\lambda$ in $\Gamma_K$, denote by $\lambda_\eqc$ the image of $\lambda$ in $\Gamma^\times_{K, \eqc}$. We apply Theorem \ref{thm:lambdasupJacprop} in the coarsened field $K_\eqc$ to obtain a decomposition $\chi: K^\ell\to \RV_\eqc^k$ such that $f$ satisfies the $\lambda_\eqc$-supremum Jacobian property on the $\lambda_\eqc$-twisted boxes of $\chi$, for every $\lambda\in \Gamma_K$. Use Lemma \ref{lem:celldecomp.Kecc.to.K} to find an $\cL$-definable cell decomposition $\phi: K^\ell\to \RV_N^n$ of depth $N$ (for some integers $N, n$) such that every $|1|_\eqc$-twisted box of $\phi$ is contained in a twisted box of $\chi$. We can moreover assume that all $\rv_\lambda(\partial f/\partial x_i)$ are constant on the $\lambda$-twisted boxes of $\phi$, for any $\lambda\in \Gamma_K^\times, \lambda\leq 1$. Fix a $\lambda\in \Gamma_K^\times, \lambda\leq 1$ and assume that $f$ does not satisfy the $|M|\lambda$-supremum Jacobian property on the $|M|\lambda$-twisted boxes of $\phi$, for any integer $M$. By saturation, we may then find a single pair $x\neq y\in K^\ell$ such that for any integer $M$, $x$ and $y$ are in the same $|M|\lambda$-twisted box of $\phi$, but such that 
\[
|f(x) - f(y) -  \grad f(y)\cdot(x - y) | \geq  |M|\lambda\cdot |\grad f(y) |\cdot |x-y|
\]
and $\grad f(y) \neq 0$. But this contradicts the conclusion of Theorem \ref{thm:lambdasupJacprop}, proving the result in mixed characteristic as well.
\end{proof}

\section{Product of ideal preparation}\label{sec:product.of.ideal.prep}

In this section we prove a preparation result with products of ideals. The idea is as follows. Suppose that $K$ is equipped with a $\ell$-h-minimal structure ($\ell\geq 1$) and let $\lambda\in \Gamma_K^\times, \lambda\leq 1$. If $X$ is a subset of $K$ which is $A$-definable, where $A$ is a subset of $\RV_\lambda$ with at most $\ell$ elements, then we may $\lambda$-prepare $X$ by a finite $\emptyset$-definable set. This is precisely the definition of $\ell$-h-minimality. Now suppose that $A$ has more than $\ell$ elements. Is it still possible to prepare the set $X$? Note that we cannot expect to be able to $\lambda$-prepare $X$ in this setting. However, we prove that if $A$ has $\ell+r$ elements, then $X$ is $\lambda^r$-prepared by a finite $\emptyset$-definable set.

We prove such a result more generally with parameters in $\RV_{I_i}$ for proper non-trivial ideals $I_1, ..., I_r$, and we prepare functions $f: K\to K$ instead of simply subsets of $K$.

\subsection{Statement and proof}

Let $B$ be a general ball in $K$, that is, an infinite set $B \ne K$ such that for all $x,x'\in B$ and all $y\in K$ with $|x-y|\leq |x-x'|$ one has $y\in B$. We denote by
\[
\rad B = \{|x-y|\mid x, y\in B\}
\]
the radius of the ball, considered as a cut in $\Gamma_K$. For $x\in K$ we will write $\rad (x)$ for the cut $\rad B_{< |x|}(0)$. Note that then $|x|\leq \lambda$ is equivalent to $\rad(x)\subset \rad B_{<\lambda}(0)$.

Let $I$ be a proper ideal in $\cO_K$, and $c \in K$. Then $x, x'\in K$ are in the same ball $I$-next to $c$ if and only if $x=x'=c$ or $x,x'\neq c$ and
\[
\left|\frac{x-x'}{x-c}\right| \in \rad(I) = |I|.
\]

We remark that even though property (T$\ell$) is defined using only $\lambda\leq 1$ in $\Gamma_K^\times$, it also implies the correct scaling for $I$-twisted boxes. In more detail, let $f: K^\ell\to K$ be a function which satisfies (T$\ell$) with respect to some decomposition $\chi: K^\ell\to \RV^k$. Let $F$ be a twisted box of $\chi$ and let $G$ be an $I$-twisted box of $\chi$ contained in $F$, for some proper non-trivial ideal $I$ of $\cO_K$. Assume that $f(F)$ is an open ball. Then $f(G)$ is a ball, and
\[
\rad f(G) = \radop (f(F)) \cdot |I|.
\]

\begin{defn}\label{defn.Iprep}
Let $I$ be a proper non-trivial ideal of $\cO_K$, $f: K\to K$ a function and $C\subset K$ a finite set. We say that $f$ satisfies the \emph{$I$-valuative Jacobian property with respect to $C$} if the following holds. For any ball $B$ $I$-next to $C$, there exists a $\mu_B\in \Gamma_K$ such that
\begin{enumerate}
\item for $x_1, x_2\in B$ we have $|f(x_1)-f(x_2)| = \mu_B|x_1-x_2|$, and
\item if $\mu_B\neq 0$ and if $B'\subset B$ is any ball then $f(B')$ is also a ball, with $\rad f(B') = \mu_B \rad B'$.
\end{enumerate}
\end{defn}

We remark that the $1$-valuative Jacobian property is simply the valuative Jacobian property from~\cite[Cor.\,3.1.6]{CHR}. Also, if $f$ is the characteristic function of a set $X$, then $f$ satisfying the $I$-valuative Jacobian property with respect to $C$ is the same as $X$ being $I$-prepared by $C$.

Note that in the following theorem, we do not need the technical condition  of $\Gamma_K$-open ideals from \cite[Def.\,4.2.1]{CHR} and \cite[Thm.\,4.2.3]{CHR}.

\begin{thm}[Preparing with product ideals]\label{thm:crit-ell-h} 
Let $\cT$ be an $\ell$-h-minimal theory of equicharacteristic zero ($\ell\geq 1$) and let $K$ be a model of $\cT$. Let $I_1\subset I_2\subset ...\subset I_r$ be proper non-trivial $\emptyset$-definable ideals in $\cO_K$ with $r\geq \ell$ and take $\xi_i\in \RV_{I_i}$ for every $i$. Put $A=\{\xi_1, ..., \xi_r\}$. Let $f: K\to K$ be an $\cL(A)$-definable function. Then there exists a finite $\emptyset$-definable $C$ such that $f$ satisfies the $I_1\cdot I_{\ell+1}\cdots I_r$-valuative Jacobian property with respect to $C$.
\end{thm}

The proof of this theorem is quite long and technical. A mixed characteristic will be deduced by compactness afterwards. Let us also mention that the fact that $\ell$-h-minimality is preserved under coarsening of the valuation as in Theorem~\ref{thm: coarsening the valuation} can also be deduced from this result. During the proof, we will first prove the following weaker version, for $r=\ell$ and when $f$ is the characteristic function of an $A$-definable set. 

\begin{lem}[$I$-preparation]\label{lem: product prep for sets}
Let $\cT$ be an $\ell$-h-minimal theory of equicharacteristic zero ($\ell\geq 1$) and let $K$ be a model of $\cT$. Let $I$ be a proper non-trivial $\emptyset$-definable ideal in $\cO_K$ and let $A\subset \RV_I$ with $\#A\leq  \ell$. If $X$ is an $\cL(A)$-definable subset of $K$, then there exists a finite $\emptyset$-definable set $C\subset K$ which $I$-prepares $X$.
\end{lem}

Using compactness, we also obtain a family version.

\begin{cor}[Uniform $I$-preparation]\label{cor: product prep for families}
Let $\cT$ be an $\ell$-h-minimal theory of equicharacteristic zero ($\ell\geq 1$) and let $K$ be a model of $\cT$. Let $I$ be a proper non-trivial $\emptyset$-definable ideal in $\cO_K$. Let
\[
W\subset K\times \RV_I^\ell\times \RV^n
\]
be $\emptyset$-definable, for some integer $n$. Then there exists a finite $\emptyset$-definable $C$ uniformly $I$-preparing $W$.
\end{cor}

\begin{proof}
This follows from Lemma \ref{lem: product prep for sets} using compactness in the same way as \cite[Prop.\,2.6.2]{CHR}.
\end{proof}

\begin{proof}[Proof of \Cref{thm:crit-ell-h}] \textbf{Step 1.} Without loss of generality, we may assume that $I_1=...=I_\ell=:I$. We induct on $r$, the base case being $r=\ell$. Write $A=\{\xi_1, ..., \xi_\ell\}$. Consider $f$ as a $\emptyset$-definable family $f_y: K\to K$, for $y\in K^\ell$ with $f_y=f$ for $y\in Y := \prod_i \rv_{I}^{-1}(\xi_i)$. By 1-h-minimality \cite[Lem.\,2.8.5]{CHR} there is for every $y\in K^\ell$ a finite $y$-definable $C_y$ such that $f_y$ satisfies the $1$-valuative Jacobian property with respect to $C_y$. By compactness we can assume that $C_y$ is a $\emptyset$-definable family with a parameter $y\in K^\ell$. We apply Lemma~\ref{lem:finite.set.to.functions} to find a $\emptyset$-definable family $(Y_\eta)_{\eta\in \RV^k}$ of subsets of $K^\ell$, together with a $\emptyset$-definable family of maps $g_\eta: Y_\eta\to K$ such that
\[
\bigcup_{y\in K^\ell} \{y\}\times C_y = \bigsqcup_{\eta\in \RV^k} \operatorname{graph} g_\eta.
\]

Using \cite[Thm.\,5.2.4, Add.\,1]{CHR} and Lemma \ref{lem: ellhmin to Tell, D}, take a decomposition $\chi: K^\ell\to \RV^m$ of $K^\ell$ with the following properties.
\begin{enumerate}
\item For every twisted box $F$ of $\chi$ and for every $\eta\in \RV^k$, either $F\subset Y_\eta$ or $F\cap Y_\eta = \emptyset$.
\item The coordinate-wise map $\rv:K^\ell\to \RV^\ell$ is constant on the twisted boxes of $\chi$. 
\item For every $\eta\in \RV^k$, the map $g_\eta$ satisfies (T$\ell$) with respect to $\chi$. 
\item For every $\eta\in \RV^k$ the $\eta$-definable map
\[
\tilde{f}_\eta: K^\ell\to K: y\mapsto f_y(g_\eta(y))
\]
satisfies (T$\ell$) with respect to $\chi$. 
\end{enumerate}
By preparing families \cite[Prop\,2.6.2]{CHR}, we find a finite $\emptyset$-definable set $C\subset K$ which $1$-prepares $g_\eta(\chi^{-1}(\zeta))$, for all $\eta\in \RV^k, \zeta\in \RV^m$.

\textbf{Step 2a.} Let us first assume that $f$ is the indicator function of some $A$-definable set $X\subset K$. We claim that then $f$ satisfies the $I$-valuative Jacobian property with respect to $C$, or equivalently, $C$ $I$-prepares $X$. Suppose not, then there is a ball $B$ $I$-next to $C$, and elements $x, x'\in B$ such that $f(x)=0$ while $f(x')=1$. Let $B_1$ be the smallest closed ball containing $x$ and $x'$ and fix $y_0\in Y$. Since $C_{y_0}$ 1-prepares $X$, $C_{y_0}\cap B_1$ cannot be empty. Therefore, there exists some $\eta\in \RV^k$ such that $y_0\in Y_\eta$ and $g_\eta(y_0)\in B_1$. Fix this $\eta$ for the rest of this step.

Let $F$ be the twisted box of $\chi$ containing $y_0$, and let $F'$ be the $I$-twisted box of $\chi$ containing $y_0$. Since $\rv: K^\ell\to \RV^\ell$ is constant on $F$, also $\rv_I$ is constant on $F'$ and hence $F'\subset Y$. If $g_\eta$ is constant on $F'$, then there exists a finite $\emptyset$-definable set $C'$ containing this singleton $g_\eta(F')$. Indeed, the set of $\zeta\in \RV^m$ for which $g_\eta$ is constant on $\chi^{-1}(\zeta)$ is $\eta$-definable, and the resulting set of values $g_\eta(\chi^{-1}(\zeta))$ is finite and $\eta$-definable. By taking the union over all $\eta\in \RV^k$ and using~\cite[Cor.\,2.6.7]{CHR}, we obtain a finite $\emptyset$-definable set $C'$ such that $f$ has the $1$-valuative Jacobian property with respect to $C'$, so certainly also the $I$-valuative Jacobian property. So we can assume that $g_\eta$ is non-constant on $F'$. We claim that $g_\eta(F')$ is a proper subset of $B_1$. Suppose towards a contradiction that this is not true. Note that $g_\eta(F')$ and $B_1$ are both balls containing $g_\eta(y_0)$. Our assumption then implies that $B_1\subset g_\eta(F')$. Property (T$\ell$) for $\tilde{f}_\eta$ tells us that $\tilde{f}_\eta(F') = f(g_\eta(F'))$ is either a ball, or a singleton. But $x,x'$ are both in $g_\eta(F')$, so that $\tilde{f}_\eta(F') = \{0,1\}$ which is not a ball. This is a contradiction and so 
\[
g_\eta(F')\subsetneq B_1.
\]
Since $C$ 1-prepares $g_\eta(F)$, there is some $c\in C$ for which
\[
|g_\eta(y_0)-c|\leq \radop g_\eta(F).
\]
We conclude, by property (T$\ell$) for $g_\eta$, that
\[
|g_\eta(y_0)-c|\cdot |I|\subset \radop (g_\eta(F))\cdot |I| = \rad g_\eta(F')\subsetneq \rad B_1.
\]
Using that $x$ and $g_\eta(y_0)$ are both in $B_1$ and that $I$ is a proper ideal gives that
\[
|x-c|\cdot |I|\subset \max\{|x - g_\eta(y_0)|, |g_\eta(y_0) - c|\}\cdot |I| \subsetneq \rad B_1.
\]
Thus we have that $|x-x'|/|x-c|\notin |I|$, since $|x-x'|$ is the largest element of $\rad B_1$. In other words, $x$ and $x'$ are not in the same ball $I$-next to $C$, contrary to assumption. 

This proves Lemma \ref{lem: product prep for sets} and hence we can use Corollary \ref{cor: product prep for families} in the rest of the proof.

\textbf{Step 2b.} Now assume that $f$ is just any $A$-definable function. By following \cite[Lem.\,2.8.2]{CHR}, we will enlarge $C$ so that $f$ is constant or injective on balls $I$-next to $C$. The set $W_\infty$ of $w\in K$ such that $f^{-1}(w)$ is infinite is a finite $A$-definable set, in view of \cite[Lem.\,2.8.1]{CHR} (or alternatively by Theorem~\ref{thm: equivalences} and property (D)). By Lemma~\ref{lem: product prep for sets} we can find a finite $\emptyset$-definable set $C_0$ which $I$-prepares $f^{-1}(w)$ for any $w\in W_\infty$. 

If $w\in K$ is not in $W_\infty$, then $f^{-1}(w)$ is finite. So by \cite[Lem.\,2.5.3]{CHR} there exists an $A$-definable family of injections $h'_w: f^{-1}(w) \to \RV^n$ for some integer $n$. Put $h_w'(x)=0$ if $w\in W_\infty$ and $x\in f^{-1}(w)$ and consider the $A$-definable function 
\[
g: K\to \RV^n: x\mapsto h'_{f(x)}(x).
\]
Use Corollary \ref{cor: product prep for families} to find a finite $\emptyset$-definable set $C_1$ which $I$-prepares this function $g$. If $B$ is a ball $I$-next to $(C_0\cup C_1)$, then $f$ is either constant or injective on $B$, see \cite[Lem.\,2.8.2]{CHR} for details. 

Using the discussion above, we now enlarge our set $C$ such that $f$ is also constant or injective on balls $I$-next to $C$. We claim that this $f$ satisfies the $I$-valuative Jacobian property with respect to $C$. Suppose not, then there is a ball $B$ $I$-next to $C$, and distinct elements $x, x', x''\in B$ such that
\begin{align*}
|f(x)-f(x')| &= \mu_1|x-x'| \\
|f(x)-f(x'')| &= \mu_2|x-x''|,
\end{align*}
for $\mu_1, \mu_2\in \Gamma_K$, $\mu_1\neq \mu_2$. If $f$ were constant on $B$ then $\mu_1=\mu_2 = 0$. So we can assume that $f$ is injective on $B$. Let $|x-x'|\geq |x-x''|$ and let $B_1$ be the smallest closed ball containing $x$ and $x'$. Fix $y_0\in Y$. Then $f$ satisfies the $1$-valuative Jacobian property with respect to $C_{y_0}$, so $C_{y_0}\cap B_1$ cannot be empty. Therefore, there exists some $\eta\in \RV^k$ such that $y_0\in Y_\eta$ and $g_\eta(y_0)\in B_1$. We fix this $\eta$ for the rest of this step.

Let $F$ be the twisted box of $\chi$ containing $y_0$, and let $F'$ be the $I$-twisted box of $\chi$ containing $y_0$. Since $\rv: K^\ell\to \RV^\ell$ is constant on $F$, also $\rv_I$ is constant on $F'$ and hence $F'\subset Y$. If $g_\eta$ is constant on $F'$, then as before there exists a finite $\emptyset$-definable set $C'$ containing this singleton $g_\eta(F')$. The map $f$ will satisfy the $1$-valuative Jacobian property with respect to $C'$, so certainly also the $I$-valuative Jacobian property. So we can assume that $g_\eta$ is non-constant on $F'$. We claim that $g_\eta(F')$ is a proper subset of $B_1$. Suppose towards a contradiction that this is not true. Note that by property (T$\ell$) $g_\eta(F')$ and $B_1$ are both balls containing $g_\eta(y_0)$. Our assumption then implies that $B_1\subset g_\eta(F')$. Define
\begin{align*}
& Z = g_\eta^{-1}(x)\cap F', Z' = g_\eta^{-1}(f(x'))\cap F' \\
& \tilde{Z} = \tilde{f}_\eta^{-1}(x)\cap F', \tilde{Z}' = \tilde{f}_\eta^{-1}(f(x'))\cap F'.
\end{align*}
By property (T$\ell$), $g_\eta(F)$ and $\tilde{f}_\eta(F)$ are open balls, say of radii $\mu_0$ and $\tilde{\mu}_0$. Define $\lambda_0 = |x-x'|/\mu_0$ and $\tilde{\lambda}_0 = |f(x)-f(x')|/\tilde{\mu}_0$. Note that 
\[
\mu_1 = \frac{\tilde{\lambda}_0}{\lambda_0}\cdot \frac{\tilde{\mu}_0}{\mu_0}.
\]
The goal is to show that $\tilde{\lambda}_0 = \lambda_0$. The element $\lambda_0$ (resp.\ $\tilde{\lambda}_0$) is the largest element of $\Gamma_K^\times$ such that for some $z\in Z$ (resp.\ $z\in \tilde{Z}$), the $\lambda_0$-twisted box $G$ of $\chi$ (resp.\ $\tilde{\lambda}_0$-twisted box) containing $z$ is disjoint from $Z'$ (resp.\ $\tilde{Z}'$). Indeed, if $G$ is the $\lambda_0$-twisted box of $\chi$ around $z$, then by (T$\ell$) the image $g_\eta(G)$ is an open ball of radius $|x-x'|$ containing $x = g_\eta(z)$. Since $Z'\subset \tilde{Z}'$ we immediately obtain that $\tilde{\lambda}_0\leq \lambda_0$. For the other inequality, fix $z\in Z$ and let $G$ be the $\lambda_0$-twisted box of $\chi$ containing $z$. Then $G$ is a subset of $F'$, since $g_\eta(G) = B_{<|x-x'|}(x) \subset B_1\subset g_\eta(F')$. Now suppose that $G$ contains an element from $\tilde{Z}'$, say $y\in \tilde{Z}'$. We have that 
\[
g_\eta(y)\in g_\eta(G) = B_{<|x-x'|}(x).
\]
On the other hand, we have that $f$ is injective on $B_1$ and $f(x') = \tilde{f}_\eta(y) = f(g_\eta(y))$. So $x' = g_\eta(y)$ which is impossible. We conclude that $\tilde{\lambda}_0 = \lambda_0$ and so $\mu_1 = \tilde{\mu}_0 / \mu_0$. However, exactly the same argument works for $\mu_2$, showing that $\mu_1 = \mu_2$. This is a contradiction and we conclude that $g_\eta(F')\subsetneq B_1$. 

The rest of the argument for this step is identical to step 2a, proving that (1) in Definition \ref{defn.Iprep} is satisfied.

\textbf{Step 3.} Let $B$ be a ball $I$-next to $C$. We now prove that $f$ maps balls contained in $B$ to balls of the correct radius, which is (2) in Definition~\ref{defn.Iprep}. If $f$ is constant then this is obvious, so assume that $f$ is injective on $B$. Let $B'\subset B$ be a ball, which is $J$-next to $C$ for some ideal $J\subset I$ of $\cO_K$. Firstly, for $x, x'\in B'$ we have that
\[
|f(x) - f(x')| = \mu |x-x'| \in \mu \rad B',
\]
for some fixed $\mu\in \Gamma_K$ (depending only on $B$). So if we let $B_1$ be the smallest ball in $K$ containing $f(B')$ then this is a ball of radius $\mu \rad B'$. Suppose towards a contradiction that $f(B')$ is a proper subset of $B_1$. In other words, $f(B')$ is not a ball. Let $y_0\in Y$. Since $f$ satisfies the $1$-valuative Jacobian property with respect to $C_{y_0}$, we must have that $C_{y_0}\cap B'$ is non-empty. So there exists $\eta\in \RV^k$ such that $y_0\in Y_\eta$ and such that $g_\eta(y_0)\in B'$. Let $G$ be the $J$-twisted box of $\chi$ containing $y_0$. By property (T$\ell$) for $g_\eta$, we have that $g_\eta(G) = B'$. Since $J\subset I$ and since $\rv$ is constant on the twisted boxes of $\chi$ we have that $G\subset Y$. But then (T$\ell$) for $\tilde{f}_\eta$ implies that 
\[
f(B') = f(g_\eta(G)) = \tilde{f}_\eta(G)
\]
is a ball. This is the desired contradiction and proves the case $r = \ell$.

\textbf{Step 4.} We now assume that $r > \ell$. Write $A = A'\cup \{\xi\}$ with $\#A' = r-1$ and $\xi = \xi_r$. Define $I = I_r$ and $J = I_1\cdot I_{\ell+1}\cdots I_{r-1}$. Then we want to find a finite $\emptyset$-definable set $f$ satisfies the $IJ$-valuative Jacobian property with respect to $C$. We again consider $f$ as a family of $A'$-definable functions $f_y: K\to K$ for $y\in K$ with $f_y = f$ if $y\in Y = \rv_I^{-1}(\xi)$. By induction, for every $y$ there is a finite $y$-definable set $C_y$ such that $f_y$ satisfies the $J$-valuative Jacobian property with respect to $C_y$. We apply Lemma~\ref{lem:finite.set.to.functions} again to find a $\emptyset$-definable family $(Y_\eta)_{\eta\in \RV^k}$ of subsets of $K$, together with a $\emptyset$-definable family of maps $g_\eta: Y_\eta\to K$.

Use preparation of families \cite[Prop.\,2.6.2]{CHR} to find a finite $\emptyset$-definable set $D\subset K$ with the following properties.
\begin{enumerate}
\item For every $\eta\in \RV^k$, the set $Y_\eta$ is 1-prepared by $D$.
\item For every $\eta\in \RV^k$, the function $g_\eta$ (extended by zero outside its domain) satisfies the $1$-valuative Jacobian property with respect to $D$.
\item For every $\eta\in \RV^k$, the $(A\cup\{\eta\})$-definable function
\[
\tilde{f}_\eta: K\to K: y\mapsto f_y(g_\eta(y))
\]
satisfies the $J$-valuative Jacobian property with respect to $D$. Such a set $D$ exists by induction and compactness.
\item Finally, we also require that $0\in D$.
\end{enumerate}
The set of balls 1-next to $D$ is parametrized by $\RV$-parameters, by \cite[Lem.\,2.5.4]{CHR}. Hence we can find a finite $\emptyset$-definable set $C$ in $K$ which 1-prepares $g_\eta(B)$, for every ball $B$ 1-next to $D$, and for every $\eta\in \RV^k$. We claim that $f$ satisfies the $IJ$-valuative Jacobian property with respect to $C$. If $Y\cap D\neq \emptyset$ then $f$ even satisfies the $J$-valuative Jacobian property with respect to the $\emptyset$-definable set $\cup_{d\in D} C_d$. So we can assume that $Y$ and $D$ are disjoint. Let $F$ be the ball $1$-next to $D$ containing $Y$. Let $B$ be a ball $IJ$-next to $C$. 

\textbf{Case 1:} Suppose first that for every $\eta\in \RV^k$ we have that $B\cap g_\eta(Y) = \emptyset$. Fix $\eta\in \RV^k$, $y_0\in Y$ and $x, x'\in B$. The hypothesis implies that 
\[
\rad (x-g_\eta(y_0)) \supset \rad g_\eta(Y).
\]
In other words, we have by the $1$-valuative Jacobian property of $g_\eta$ that
\[
|x-g_\eta(y_0)| \notin \rad g_\eta(Y) = |I|\radop(g_\eta(F)).
\]
By construction, the set $C$ 1-prepares $g_\eta(F)$, so there exists some $c_0\in C$ such that
\[
|c_0-g_\eta(y_0)|\leq \radop (g_\eta(F)).
\]
Combined with the above, we obtain that $|x-g_\eta(y_0)| \notin |I||c_0-g_\eta(y_0)|$. Since $I$ is a proper ideal, the triangle inequality yields that for every $\lambda \in I$
\[
\lambda|x-c_0| \leq \max\{ \lambda |x-g_\eta(y_0)|, \lambda|g_\eta(y_0)-c_0|\} \leq |x-g_\eta(y)|.
\]
Therefore, since $x$ and $x'$ are in the same ball $IJ$-next to $C$ we obtain that
\[
|x-x'|\in |IJ||x-c_0| \subset |J||x-g_\eta(y_0)|,
\]
and $x,x'$ are in the same ball $J$-next to $g_\eta(y_0)$. Since this argument worked for every $\eta\in \RV^k$, we conclude that $B$ is contained in a ball $J$-next to $C_{y_0}$ (for any $y_0\in Y$). But $f=f_{y_0}$ is prepared on such a ball by our initial construction of $C_{y_0}$.

\textbf{Case 2:} Suppose now that $g_\eta(y_0)\in B$ for some $\eta\in \RV^k$ and some $y_0\in Y$. Since $g_\eta(Y)$ is $I$-prepared by $C$, $B$ is contained in $g_\eta(Y)$. If $g_\eta(Y)$ is a singleton, then also $B$ is a singleton and there is nothing to prove. So we can assume that $g_\eta$ is injective on $F$. Let $Y'\subset Y$ be a ball $IJ$-next to $D$ such that $g_\eta(Y')\cap B\neq \emptyset$. Since $C$ $IJ$-prepares $g_\eta(Y')$ we have that $B\subset g_\eta(Y')$. Now $D$ prepares both $g_\eta$ and $\tilde{f}_\eta$ on $Y'$, so there exist $\mu_{g_\eta}, \mu_{\tilde{f}_\eta}\in \Gamma_K$ such that for $y, y'\in Y'$ we have that
\begin{align*}
|g_\eta(y)- g_\eta(y')| &= \mu_{g_\eta} |y-y'| \\
|\tilde{f}_\eta(y) - \tilde{f}_\eta(y')| &= \mu_{\tilde{f}_\eta}|y-y'|.
\end{align*}
Note that $\mu_{\tilde{g_\eta}}$ is non-zero, as $g_\eta$ is injective on $Y'$.
Let $x, x'\in B$ and take $y, y'\in Y'$ be such that $g_\eta(y) = x$ and $g_\eta(y')=x'$. Then
\begin{align*}
|f(x)-f(x')| &= |f(g_\eta(y)) - f(g_\eta(y'))| = |\tilde{f}_\eta(y) - \tilde{f}_\eta(y')| \\
	&= \mu_{\tilde{f}_\eta}|y-y'| = \frac{ \mu_{\tilde{f}_\eta} }{ \mu_{g_\eta} }|x-x'|.
\end{align*}
This is the desired situation, proving that $f$ satisfies (1) of Definition \ref{defn.Iprep} on $B$.

Proving that $f$ maps balls contained in $B$ to balls of the correct radius is identical to the previous argument for $r=\ell$.
\end{proof}

\subsection{Some corollaries}

In mixed characteristic, compactness yields the following analogue.

\begin{cor}\label{cor:mixed.char.product.prep}
Let $\cT$ be an $\ell$-\hc-minimal theory, $\ell\geq 1$, and let $K$ be a model of $\cT$. Let $I_1\subset I_2\subset ...\subset I_r$ be proper $\emptyset$-definable ideals in $\cO_K$ with $r\geq \ell$ and take $\xi_i\in \RV_{I_i}$ for every $i$. Put $A=\{\xi_1, ..., \xi_r\}$. Let $f: K\to K$ be an $\cL(A)$-definable function. Then there exists a finite $\emptyset$-definable $C$ and a positive integer $M$ such that for every ball $B$ which is $|M|I_1\cdot I_{\ell+1}\cdots I_r$-next to $C$, there exists a $\mu_B$ such that for $x, x'\in B$ we have that
\[
\mu_B|M||x-x'|\leq |f(x)-f(x')|\leq \mu_B |M|^{-1}|x-x'|.
\]
\end{cor}
\begin{proof}
We may assume that $K$ is $\aleph_0$-saturated. If $I$ is any ideal of $\cO_K$, denote by $c(I) = \cap_M |M|I$ the coarsened ideal in $\cO_{K,\eqc}$. Using Theorem \ref{thm:crit-ell-h} in the coarsened structure $K_\eqc$, we obtain a finite $\cL_\eqc$-definable set $C'$ such that for any ball $B$ which is $c(I_1\cdot I_{\ell+1}\cdots I_r)$-next to $C'$, there exists a $\mu_B\in \Gamma_{K,\eqc}$ such that for any $x, x'\in B$ we have
\begin{equation}\label{eq:T1.in.coarsened}
|f(x)-f(x')|_\eqc = \mu_B |x-x'|_\eqc.
\end{equation}
By Lemma \ref{lem:celldecomp.Kecc.to.K} there exists a finite $\cL$-definable $C$ which contains $C'$. (Alternatively, one can rely on \cite[Lem.\,2.5.3]{CHRV} to find such a $C$.)
Now assume that the desired result is false for any positive integer $M$. By saturation of $K$, we may find a single pair $(x,x')$ in $K^2$ such that $x$ and $x'$ are in the same ball $|M|I_1\cdot I_{\ell+1}\cdots I_r$-next to $C$, for every $M$, but such there is no $M$ for which
\begin{equation*}
\mu_B|M||x-x'|\leq |f(x)-f(x')|\leq \mu_B |M|^{-1}|x-x'|,
\end{equation*}
holds (here $\mu_B$ should be considered in $\Gamma_K$). But then $x$ and $x'$ are in the same ball $c(I_1\cdot I_{\ell+1}\cdots I_r)$-next to $C$, contradicting Equation~\eqref{eq:T1.in.coarsened}.
\end{proof}

We mention a final corollary about the Jacobian property for functions defined by using many $\RV_\lambda$-parameters. This Jacobian property is a stronger version of the $I$-valuative Jacobian property from above.

\begin{defn}[$\RV_\lambda$-Jacobian property]\label{defn:RVjac.prop}
Let $\lambda$ be in $\Gamma_K^\times$ with $\lambda<1$. Say that a function $f:B\to K$ with $B\subset K$ a ball has the $\RV_\lambda$-Jacobian property if the following hold:
\begin{enumerate}
 \item The derivative $f'$ exists on $B$ and $\rv_\lambda\circ f'$ is constant on $B$.
 \item $\rv_\lambda(\frac{f(x) - f(y)}{x - y}) = \rv_\lambda(f')$ for all $x,y$ in $B$.
 \item For every open ball $B' \subset B$, $f(B')$ is either a point or an open ball.
\end{enumerate}
\end{defn}

\begin{cor}\label{cor:RV-jac-prop}
Suppose that $\cT$ is $1$-\hc-minimal, and let $K$ be a model of $\cT$. Consider $\lambda\in\Gamma_K^\times$ with $\lambda<1$, a set $A\subset \RV_\lambda$ and an $\cL(A)$-definable function $f:X\subset K\to K$. Then there are positive integers $m,n$ and a finite $\emptyset$-definable $C \subset K$ such that $X$ is $|m|\lambda^{n}$-prepared by $C$ and such that the restriction of $f$ to $B$ has the $\RV_\lambda$-Jacobian property, for each ball $B$ in $X$ which is $|m|\lambda^{n}$-next to $C$. 
\end{cor}
\begin{proof}
That $f'$ exists away from a set like $C$ follows from \cite[Cor.\,3.1.4]{CHRV}. That both $f$ and $X$ can be $|m|\lambda^{n}$-prepared by $C$ for such $m,n,C$ with furthermore properties (1) and (3) from \Cref{defn:RVjac.prop} for $f$ follows from \Cref{cor:mixed.char.product.prep} and compactness. So, it only remains to ensure (2) from \Cref{defn:RVjac.prop}. But (2) follows by making $m$ and $n$ larger if necessary. Indeed, (2) follows from (1) and (3) on strictly smaller balls, and, making $n$ and $m$ larger creates smaller balls. 
\end{proof}

\section{Some further questions}\label{sec:further.questions}

In this section we record some natural follow-up questions. Some of these already appear in \cite{CHRV}, but they are relevant to the contents of this article as well.

The most immediate question is about distinguishing between the various notions of $\ell$-h-minimality.

\begin{question}
Are the notions of 1-h-minimality and $\omega$-h-minimality equivalent? How about 0-h-minimality and 1-h-minimality? Or more generally, are there any $\ell<\ell'$ for which $\ell$-h-minimality and $\ell'$-h-minimality are equivalent?
\end{question}

In a different direction, one can ask many questions about 0-h-minimality, and by extension 0-\hc-minimality.

\begin{question}
Is there an analytic criterion for 0-h-minimality, similar to property (T$\ell$) for $\ell$-h-minimality? Or perhaps for 0-\hc-minimality?
\end{question}

\begin{question}
Are 0-h-minimality and 0-\hc-minimality equivalent? And how much of the general theory of Hensel minimality works under 0-\hc-minimality?
\end{question}

In \cite{CHRV} a mixed characteristic analytic criterion is given for 1-h-minimality. In this article, we gave such a criterion in equicharacteristic zero only. 

\begin{question}
What is the correct property (T$\ell$) in mixed characteristic? And is it (together with property (D)) equivalent to $\ell$-h-minimality?
\end{question}

In Lemma \ref{lem:uniform.prep.ellhmin} we saw that $\ell$-\hc-minimality implies $\ell$-\hmix-minimality. In \cite{CHRV} it is proved that these are equivalent for $\ell=1$, by passing through a mixed characteristic version of property (T1, D). 

\begin{question}
Let $\ell\geq 1$. Is $\ell$-\hmix-minimality equivalent to $\ell$-\hc-minimality? 
\end{question}

The following question is also natural, about preparing functions $f: K^n\to K^m$ for $n,m\geq 1$ as in Theorem~\ref{thm.f.ell.III}.

\begin{question}
Is there some form of property (T$\ell$), or even of the compatible domain and image preparation Theorem \ref{thm.f.ell.III}, for functions $f: K^\ell\to K^m$ under $\ell$-h-minimality?
\end{question}

\bibliographystyle{amsplain}
\bibliography{anbib}

\end{document}